\documentclass[12pt]{amsart}

\usepackage{amssymb}
\usepackage{amsthm}
\usepackage{amsmath}
\usepackage{textcomp}
\usepackage{mathrsfs}
\usepackage{tikz}
\usepackage[colorlinks,linkcolor=blue,citecolor=red]{hyperref}
\usepackage{booktabs}

\newtheorem{propo}{Proposition}[section]
\newtheorem{theo}[propo]{Theorem}
\newtheorem{lemma}[propo]{Lemma}
\newtheorem{corol}[propo]{Corollary}
\newtheorem{prob}[propo]{Problem}
\newtheorem{defi}[propo]{Definition}

\newtheorem{pro}[propo]{Proposition}

\newtheorem{exam}[propo]{Example}

\newcommand{\stack}[2]{\genfrac{}{}{0pt}{}{#1}{#2}}

\def\Aut{{\sf Aut}}

\def\dim{{\sf dim}\,}
\def\GL{{\rm GL}}
\def\PGL{{\rm PGL}}

\def\Sym{{\rm Sym}}

\def\rk{{\rm rk}}
\def\d{{\rm d}}
\def\t{{\sf t}}

\def\o{{\omega}}
\def\Ga{\varGamma}
\def\M{{\sf M}}
\def\D{{\sf D}}
\def\l{\langle}
\def\r{\rangle}
\def\PGaU{{\rm P\Gamma U}}
\def\PGaL{{\rm P\Gamma L}}
\def\M{{\sf M}}
\def\Z{{\sf Z}}
\def\PG{{\rm PG}}
\def\J{{\rm J}}
\def\G{{\rm G}}
\def\DD{{\rm D}}
\def\H{{\rm H}}
\def\BF{{\rm BF}}
\def\AF{{\rm AF}}
\def\HF{{\rm HF}}
\def\O{{\rm O}}
\def\GH{{\rm GH}}
\def\GO{{\rm GO}}
\def\GD{{\rm GD}}
\def\K{{\rm K}}
\def\vs{\vskip0.05in}

\begin{document}

\title{Geodesic-transitive graphs with large diameter}
\thanks{This work is based upon the author's thesis \cite{thesis}.}
\author[P. C. Hua]{Pei Ce Hua}
\address{Peice Hua\\Shenzhen International Center for Mathematics\\
	Southern University of Science and Technology\\
	Shenzhen, 518055, P. R. China}
\email{huapc@pku.edu.cn}

\date\today

\begin{abstract}
We review the nearly complete classification project for finite distance-transitive graphs and compile a list of all known graphs. Interestingly, we find that those graphs with diameter larger than $ 4 $, apart from a small finite number of exceptions, are geodesic-transitive. Their geodesics exhibit a clear (often geometric) structure. On the other hand, we provide examples of graphs that are distance-transitive but not geodesic-transitive, including two
infinite families with diameter $ 3 $ and a few sporadic ones with diameter $ 3 $, $ 4 $ or $ 7 $. In the last section, we extend our investigation to polar Grassmann graphs and provide an explicit description of their geodesics.
\end{abstract}

\maketitle	

\begin{flushleft}
	

\end{flushleft}

%\begin{flushleft}

%{\bf Math. Subj. Class.:} 05E18; 20B25.

%\end{flushleft}

\section{Introduction}\label{sec-intro}

This paper is part of a series aimed at investigating the following problem\,:
\begin{prob}\label{prob11}
{\it Identify which distance-transitive graphs are geodesic-transitive.}
\end{prob}

All the graphs $ \Ga $ we consider are finite, simple, undirected, and connected. The \textit{distance} $ \d(X,Y) $ for any two vertices $ X,Y $ of $ \Ga $ is defined to be the length of a shortest path connecting $ X $ with $ Y $, where such a path is called a \textit{geodesic}. The \textit{diameter} $ \d $ of $ \Ga $ is the maximum distance in $ \Ga $. Depending on the context, a \textit{diameter} of $ \Ga $ is also used to refer to a longest geodesic of length $ \d $.\vs\vs\vs

\subsection{\it Three distance-related types of graph symmetry}{~}\vskip0.03in

A graph is called a \textit{distance-regular graph}, if there is a sequence of positive integers $ \{b_0,b_1\cdots,b_{\d-1}\,;c_1,c_2,\cdots,c_{\d}\} $ such that for any two vertices $ X,Y $ at distance $ i $, there are precisely $ c_i $ (resp.~$ b_i $) neighbors of $ X $ at distance $ i-1 $ (resp.~$ i+1 $) with $ Y $, where $ \d $ is the diameter. The sequence is called the \textit{intersection array}, and these integers are called the \textit{intersection numbers}. The study of distance-regular graphs has become a classic topic with a substantial body of results\,; refer to the book \cite{DRG-book} of the same title authored by A.E. Brouwer et al. \vs

One intriguing 
phenomenon is\,: this arithmetical regularity property often (especially in cases of large diameter) implies the existence of stronger symmetries. A graph is called \textit{$ G $-distance-transitive}, if it has an automorphism group $ G $ which acts transitively on the set of ordered vertex-pairs at each given distance. We may omit the prefix “$ G $-” if it is not specified. A distance-transitive graph is, of course, distance-regular. A notable fact observed in \cite{DRG-book} is that most known distance-regular graphs with diameter $ \d>3 $ are distance-transitive.\vs

The classification project for distance-transitive graphs, after a huge amount of effort, has reached its final stage. Interestingly, we observe a similar phenomenon regarding these graphs\,: those with large diameter often exhibit stronger symmetry. A graph is called \textit{$ G $-geodesic-transitive}, if it has an automorphism group $ G $ which acts transitively on the set of geodesics of each given length. Such graphs can be considered as a subclass of distance-transitive graphs possessing the highest “degree of symmetry”. Our first result indicates that, most known distance-transitive graphs with diameter $ \d>4 $ are geodesic-transitive.\vs\vs\vs

\subsection{\it Results}{~}\vskip0.03in

The specific study of geodesic-transitive graphs appears to be a relatively recent topic, initiated by the work \cite{GTG} in which some well-known graphs were shown to be geodesic-transitive, including Johnson graphs, Hamming graphs, Odd graphs, generalized polygons of order $ (1,t) $, and two more graphs. Earlier, in our previous work \cite{Hua} the (doubled) Grassmann graphs were studied.\vs

In Section \ref{sec-dtg}, we review the classification project for distance-transitive graphs and compile a list (Table \ref{tab-main}\,-\,\ref{tab-remaining}) of all known graphs. Note that graphs with diameter $ \d>4 $ are almost included in Table \ref{tab-main}; only a finite few are in Table \ref{tab-remaining}. In Section \ref{sec-gtg}, we show that the graphs in Table \ref{tab-main} are all geodesic-transitive\,; for graphs already addressed in \cite{Hua,GTG} we provide a short, non-inductive proof along with a refined description of their geodesics. We remark that the problem is not particularly difficult in most cases, as these distance-transitive graphs have been studied in detail, with more or less information available about their geodesics.

\begin{theo}\label{thm-gtg}
	Each graph listed in Table {\rm\ref{tab-main}} is geodesic-transitive.
\end{theo}

\begin{table}[hbt]
\renewcommand{\arraystretch}{1.2}
\newcommand{\tabincell}[2]{\begin{tabular}{@{}#1@{}}#2\end{tabular}}
\caption{{Geodesic-transitive graphs}}\label{tab-main}
\centering\scalebox{0.85}{
\begin{tabular}{lcccc}
	\toprule[1pt]
	Name & $\Ga$ & $ \d_\Ga $ (resp.) & Constructed on & \S \\
	\midrule[0.5pt]
	
	Johnson family & \tabincell{c}{${\J}(n,k) $,\, $ \bar \J(2k,k) $, \\ $ \mathcal{O}(k) $,\, $ {\O}(k) $} & \tabincell{c}{$ k $,\, $\lfloor k/2\rfloor$, \\ $2k-1$,\, $ k-1 $} & \tabincell{c}{set \\ (abstract)} & \ref{sec-set} \\
	\midrule[0.5pt]
			
	Grassmann family & $ {\G}(n,k,q) $,\, $ \mathcal{G}(k,q) $ & $ k $,\, $2k-1$ & projective space & \ref{sec-prospace}\\
	\midrule[0.5pt]
			
	\tabincell{l}{Incidence graph of \\ design {\sf D} or opposites {\sf D}$ ^{op} $ } & $ {\rm I}({\sf D}) $,\, $ {\rm I}(\D^{op\,(w,H)}) $ & $ 3 $, $ 4 $ & projective space & \ref{sec-prospace}\\
    \midrule[0.5pt]
			
	(half) Dual polar graph & $ {\DD}(W) $,\,  $\frac{1}{2}\DD(\Omega^+(2\o,q))$ & $ \o $,\, $ \lfloor\o/2\rfloor $ & polar space & \ref{sec-polarspace} \\
	\midrule[0.5pt]
			
	Hamming family & \tabincell{c}{$ {\H}(k,m) $,\, $ \frac{1}{2}{\H}(k,2) $ \\ $ \overline \H(k,2) $,\, $ \frac{1}{2}\overline \H(2k_0,2) $} & \tabincell{c}{$ k $,\, $ \lfloor k/2\rfloor $ \\ $ \lfloor k/2\rfloor $,\, $ \lfloor k_0/2\rfloor $} & tuples & \ref{sec-tuples}\\
	\midrule[0.5pt]
			
	Forms graph & \tabincell{c}{${\BF}(m,k,q) $,\\ $ {\AF}(k,q) $,\, $ {\HF}(k,r) $} & \tabincell{c}{$ \min\,\{m,k\} $,\\ $ \lfloor k/2\rfloor $,\, $ k $} & matrix space & \ref{sec-matirx} \\ 
	\midrule[0.5pt]
			
	\tabincell{l}{Generalized $ n $-gon \\ ($ n\geqslant 6 $)} & \tabincell{c}{six families \\ and their duals \\ } & $ 3$, $ 4 $ or $ 6 $ & \tabincell{c}{certain group \\ of Lie type} & \ref{sec-gp}\\
	\midrule[0.5pt]
			
	Cycles & $ {\rm C}_k $ & $ \lfloor k/2\rfloor $ & -  & - \\
	\bottomrule[1pt]	
\end{tabular}}
\end{table}

\begin{table}[htb]
	\renewcommand{\arraystretch}{1.2}
	\newcommand{\tabincell}[2]{\begin{tabular}{@{}#1@{}}#2\end{tabular}}
	\centering
	\caption{Remaining distance-transitive graphs (known)}\label{tab-remaining}
	\scalebox{0.85}{
		\begin{tabular}{l|c|c|c}
			\toprule[1pt]
			$ \Ga $  & $ \d_\Ga $ & \tabincell{c}{Geodesic\,- \\ transitivity} & Reference \\
			\midrule[1pt]
			
			Complete multipartite graph $ \K_{r,\ldots,r} $, $r\geqslant 2$ & $ 2 $ & $ \surd $ & - \\
			\midrule[0.5pt]
			
			Complete graph $ \K_n $, $n\geqslant 1$ & $ 1 $ & $ \surd $ & - \\
			\midrule[0.5pt]
			\midrule[0.5pt]
			
			Taylor graphs of unitary type $ {\rm T_1}(q) $, $ {\rm T_2}(q) $ & $ 3 $ & $\times$ & Sect.\,\ref{sec-nonGTG}\\
			\midrule[0.5pt]
			
			\tabincell{l}{Paley graphs {\rm P}$(q)$, $ q>9 $, \\ Peisert graphs {\rm Pei}$(q)$, $ q>9 $} & $ 2 $ & $\times$ & \cite[Thm.\,1.2]{GTG}\\
			\midrule[0.5pt]
			
			Some sporadic graphs & $ 3$, $4$ or $7 $ & $\times$ & Sect.\,\ref{sec-nonGTG}\\
			\midrule[0.5pt]
			\midrule[0.5pt]
			
			\tabincell{l}{Antipodal covers of $ \K_{r,r} $, $r\geqslant 3$ \\ (other than $ {\rm I}(\D^{op\,(w,H)}) $ in Table \ref{tab-main})} & $ 4 $ & ? & \cite{AKnn}\\
			\midrule[0.5pt]
			
			${\rm E}_{7,7}(q) $-graphs, affine $ {\rm E}_6(q) $-graphs & $ 3 $ & ? & \cite[\S\,10.\,7\,-\,8]{DRG-book}\\
			\midrule[0.5pt]
			
			\tabincell{l}{Incidence graphs of certain designs \\ (other than $ {\rm I}({\sf D}) $ in Table \ref{tab-main})} & $ 3 $ & ? & \cite[Thm.\,3.2]{ABd1}\\
			\midrule[0.5pt]
			
			\tabincell{l}{Antipodal covers of $ \K_{n} $,  $n\geqslant 4 $ \\ (other than $ {\rm T_1}(q) $, $ {\rm T_2}(q) $)} & $ 3 $ & ? & \cite{AKn}\\
			\midrule[0.5pt]
			
			\tabincell{l}{Remaining infinite families of \\ (primitive) rank-$ 3 $ graphs}  & $ 2 $ & ? & \cite{SRG-book}\\	
			\midrule[0.5pt]
			
			Remaining sporadic graphs & $ 2\leqslant\d_\Ga\leqslant8 $ & ? & \cite{Ad2,Bd2,ABd3}\\
			\bottomrule[1pt]
	\end{tabular}}
\end{table}

On the other hand, examples of distance-transitive graphs that are not geodesic-transitive do exist and it is interesting to find them. The first two examples known to us are the infinite families of Paley and Peisert graphs \cite{GTG}. These graphs are with diameter $ 2 $. In Section \ref{sec-nonGTG} we provide more examples which are with diameter $ 3 $, $ 4 $ or $ 7 $, including two infinite diameter 3 families. We have not yet found examples with diameter $ 5 $, $ 6 $ or $ 8 $. This may be advanced by checking computationally the remaining sporadic graphs in Table \ref{tab-remaining}.

\begin{propo}
There exist graphs with diameter $ \d\in\{2,3,4,7\} $ that are distance-transitive but not geodesic-transitive{\rm\,;} in particular, there are infinitely many with $ \d=2 $ or $ 3 $.
\end{propo}

In Section \ref{sec-pgg} we study polar Grassmann graphs which extend Grassmann graphs and dual polar graphs. These graphs have attracted many interest and exhibit close connections with many other objects, say polar Grassmannians and Tit's buildings of classical types (e.g., see \cite{pgg1,pgg2}). The second main result is as follows.

\begin{theo}\label{thm-pg}
	Let $ \Ga=\PG_{W}(k) $ be a polar Grassmann graph, $ 1\leqslant k\leqslant \o $. Then the following statements are equivalent. 
	\begin{itemize}
		\item[\rm(i)] $ \Ga $ is geodesic-transitive, \vs
		\item[\rm(ii)] $ \Ga $ is distance-regular,\vs
		\item[\rm(iii)] $ \Ga $ is a polar graph $ (k=1) $, or a dual polar graph $ (k=\o) $.
	\end{itemize}
Moreover, the geodesics of $ \Ga $ are described in {\rm Proposition \ref{prop-pg-geo}}, while their $ \Aut\,\Ga $-orbits are given in {\rm Propositions \ref{prop-pg-op}\,-\,\ref{prop-pg-orbitsnum}}.
\end{theo}

\section{Distance-transitive graphs}\label{sec-dtg}

We review the nearly complete classification project for finite distance-transitive graphs and compile the list (Table \ref{tab-main}\,-\,\ref{tab-remaining}) of all known graphs in Section \ref{sec-pri}\,-\,\ref{sec-collction}. This list is already very close to being the final, complete version. We provide several examples of non-geodesic-transitive graphs in Section \ref{sec-nonGTG}.\vs

\subsection{\it Primitivity and imprimitivity}\label{sec-pri}{~}\vskip0.03in

Let $\Ga$ be a graph with diameter $ \d $. The \textit{distance-$ i $-graph} $ \Ga^{(i)} $ of $ \Ga $ is defined on the same vertex set where two vertices are adjacent whenever they are at distance $ i $ in $ \Ga $. If $ \Ga^{(i)} $ is connected for each $ 1\leqslant i\leqslant \d $, then $ \Ga $ is called \textit{primitive}\,; otherwise, it is called \textit{imprimitive}. These notions are closely related to primitive permutation groups. Let $ G $ be an automorphism group of $ \Ga $. If $ \Ga $ is imprimitive where some $ \Ga^{(i)} $ is unconnected, then the connected components of $ \Ga^{(i)} $ give rise to an imprimitive block system of the action of $ G $ on vertices. Conversely, if $ G $ is imprimitive on the set of vertices, with some imprimitive block $ B $ which contains vertices say $ u $, $ v $ at distance $ i $.  Suppose further that  $ G $ is distance-transitive on $ \Ga $. Then the stabilizer $ G_u $ is transitive on the set $ \Ga_i(u) $ of vertices at distance $ i $ with $ u $. It then follows that $ \Ga_i(u) $, and so the connected component of $ \Ga^{(i)} $ that contains $ u $, is included in $ B $. Thus $ \Ga^{(i)} $ is unconnected and $ \Ga $ is imprimitive. In conclusion, we have the following result.

\begin{lemma}
Let $ \Ga $ be a $ G $-distance-transitive graph. Then $ G $ is primitive (on vertices) if and only if $ \Ga $ is primitive.
\end{lemma}

Every connected bipartite graph $ \Ga $ with diameter $ \d\geqslant 2 $ is imprimitive. The graphs induced by $ \Ga^{(2)} $ on its two connected components are called \textit{halved graphs} of $ \Ga $, denoted by $ \Ga^+ $ and $ \Ga^- $ (or $\frac{1}{2}\Ga$ for an arbitrary one). The graph $\frac{1}{2}\Ga$ has diameter $\lfloor \d/2\rfloor$. In some literature (e.g.,\cite{Bd2}), $ \Ga $ is called a \textit{$ B $-double} of $\frac{1}{2}\Ga$.\vs

Every \textit{antipodal graph} $ \Ga $ with diameter $ \d\geqslant 2 $  is imprimitive, since by definition the graph $ \Ga^{(\d)} $ is a disjoint union of some cliques. The antipodal quotient $ \overline\Ga $ is usually called a \textit{folded graph}, which has diameter $\lfloor \d/2\rfloor$. In the case where $ \d\geqslant3 $, $ \Ga $ and $ \overline\Ga $ have the same valency, and $ \Ga $ is called an \textit{antipodal ($ r $-) cover} of $ \overline\Ga $.\vs

\subsection{\it The compilation of \,{\rm Tables \ref{tab-main}\,-\,\ref{tab-remaining}}}\label{sec-collction}{~}\vskip0.03in

As a consequence of the famous theorem by Derek Smith \cite{Smith-reduction} (1971), the classification project for distance-transitive graphs naturally breaks into two parts\,: primitive and imprimitive. Each part has been studied in rich detail across numerous references. Here, we primarily review some recent, summarizing works, thereby avoiding the need to cite a large body of literature (which can be found in the references of these works).\vs
 
A clear and detailed version of Smith's Theorem with complete proof was given by Alfuraidan and Hall in \cite[Thm.\,2.9]{ABd1}, which set up a classification scheme. We restate this result in a concise form as follows\,:

\begin{pro}{\rm(\cite[Theorem~2.9]{ABd1})}\label{pro22}
Let $ \Ga $ be a distance-transitive graph of diameter $ \d $ and valency $ k $. Then $ \Ga $ belongs to one of the following twelve classes{\rm\,:}\vs

{\rm (i)} the generic case that consists of primitive graphs with $ \d\geqslant2 $ and $ k\geqslant 3\,; $\vs

{\rm(ii)} four canonical imprimitive cases that reduce to the generic case after halving at most once and folding at most once {\rm(}graphs in these cases are with $ \d\geqslant 4)\,; $\vs

{\rm(iii)} seven exceptional cases in which the graphs are all known, that is, either $ \Ga $ is a cycle $ {\rm C}_n\,(n\geqslant 3) $, a complete graph $ \K_n\,(n\geqslant 1) $, a complete multipartite graph $ \K_{r,\ldots,r}\,(r\geqslant 2) $, or one of the followings holds{\rm \,:}\vs

{\rm(iii.1)} $ \Ga $ is the incidence graph of a nontrivial self-dual (symmetric) $ 2 $-transitive design, $ \d=3 $ {\rm (}refer to \cite[Thm.\,3.2]{ABd1}{\rm )\,;}\vskip0.03in

{\rm(iii.2)} $ \Ga $ is an antipodal cover of $ \K_{n}\,(n\geqslant 4) $, $ \d=3 $ {\rm(}refer to~\cite{AKn}{\rm )\,;}\vskip0.03in

{\rm(iii.3)} $ \Ga $ is an antipodal cover of $ \K_{r,r}\,(r\geqslant 3)$, $ \d=4 $ {\rm(}refer to \cite{AKnn}{\rm )\,;}\vskip0.03in

{\rm(iii.4)} $ \Ga $ is the $ 6 $-cube $\H(6,2) $, $ \d=6 $ {\rm(}refer to \cite[Thm.\,3.3]{ABd1}{\rm ).}
\end{pro}

The generic case (i) is further divided into two subcases\,: $ \d=2 $ and $ \d\geqslant 3 $. The graphs with diameter $ \d=2 $ are exactly the primitive \textit{rank\,-$ 3 $~graphs} which are arisen from primitive groups with permutation rank $ 3 $. These groups have been completely classified over a long sequence of papers \cite{rank3-1,rank3-2,rank3-3,rank3-4}, while all such graphs can be found in book \cite{SRG-book}. On the other hand, the classification of graphs with diameter $ \d\geqslant 3 $ is NOT yet complete but is very close. The status is surveyed by J. van Bon \cite{pri-DTG2}, where the only unresolved portion has been reduced to specific problems concerning certain simple groups of Lie type.\vs

For the canonical imprimitive cases (ii), the latest advances can be found in a series of summary works \cite{Ad2}\,-\,\cite{im-add} by M.R. Alfuraidan et al. Specifically, the results are presented in three tables\,: \cite[Table 1]{Ad2}, \cite[Table 1]{Bd2} (for $ \d=2 $), and \cite[Table I\,]{ABd3} (for $ \d\geqslant 3 $), which list all known graphs in the generic case (i) and the associated imprimitive distance-transitive graphs they might have\footnote{~It appears that a few graphs with diameter $ 2 $ have been omitted from these tables, say the Paley and the Peisert graphs.}.\vs

Therefore, the list of all known distance-transitive graphs should be the sum of the above three tables, along with the graphs from the exceptional cases (iii). To avoid excessive length, we provide a summary of these tables rather than their full content. The graphs in \cite[Table I\,]{ABd3} are classified as follows\,:\vs
\begin{itemize}
	\item[(1)] the infinite families, including the Hamming, Johnson, Grassmann families, the (half) dual polar graphs, the forms graphs, the generalized polygons, the ${\rm E}_{7,7}(q) $-graphs, and the affine $ {\rm E}_6(q) $-graphs\,;\vskip0.03in
	
	\item[(2)] the sporadic graphs, including about $ 20 $ primitive ones and $ 6 $ imprimitive ones, among which the one with the largest diameter is the Faradjev–Ivanov–Ivanov graph with diameter $ 8 $ (whose folded graph is the Doubly Truncated Witt Graph $[\M _{22}.2] $)\,;\vs
\end{itemize}

\noindent meanwhile, the graphs in \cite[Table 1]{Ad2} and \cite[Table 1]{Bd2} are classified as follows\,:\vs

\begin{itemize}
	\item[(3)] the infinite families which can be integrated into the above class (1)\,; (Note\,: the only infinite families in these two tables which have associated imprimitive graphs ($ B $-doubles) are triangular graphs, Grassmann graphs and half dual polar graphs. Their $ B $-doubles, i.e., doubled odd/Grassmann graphs and dual polar graphs, already appear in class (1).)\vskip0.03in
	
	\item[(4)] the remaining infinite families of primitive rank-$ 3 $ graphs, say Latin squares, polar graphs, $ {\rm E}_{6,1} $-graphs and so on\,;\vskip0.03in
	
	\item[(5)] the sporadic graphs, including over 60 primitive ones and 20 imprimitive ones, among which the one with the largest diameter is the Foster graph with diameter $ 8 $ (whose halved graph is an antipodal $ 3 $-cover of $\overline{\rm T}(6)$).\vs
\end{itemize}

Finally, we collect the graphs from Proposition \ref{pro22} (iii) and those from the above classes (1)\,-\,(5) in Tables \ref{tab-main}\,-\,\ref{tab-remaining}.\vs

\subsection{\it Examples of non-geodesic-transitive graphs}\label{sec-nonGTG}{~}\vskip0.03in

The first two examples of distance-transitive graphs that are not geodesic-transitive were provided in \cite{GTG}\,: the infinite families of Paley and Peisert graphs (with three exceptions excluded). These graphs are with diameter $ 2 $. In this part, we provide two infinite families with diameter $ 3 $ (Example \ref{exam-inf}) and a few sporadic graphs with diameter $ 3 $, $ 4 $, or $ 7 $ (Example \ref{exam-spora}).\vs

In the following examples, denote the number of vertices in the graph by $ v $ and the number of geodesics of length $ i $ by $ \mathcal{L}_i $. Recalling the intersection array of a distance-regular graph, the following result follows naturally by definition.

\begin{lemma}
Let $ \Ga $ be a distance-regular graph with $ \{b_0,b_1\cdots,b_{\d-1}\,;c_1,c_2,\cdots,c_{\d}\} $ the intersection array. Then $\mathcal{L}_i=v b_0b_1\cdots b_{i-1} $, for $ 1\leqslant i\leqslant \d $.
\end{lemma}

\begin{exam}\label{exam-inf}
A \textbf{Taylor graph} is a distance-regular antipodal $ 2 $-cover of a complete graph $ K_{k+1} $ with diameter $ 3 $ and intersection array $ \{k,\mu,1\,;1,\mu,k\} $. According to \cite[\S\,7.6.C (v), pp.\,228-229]{DRG-book} {\rm(}see also {\rm Taylor} \cite[Thm.\,1\,\&\,2]{Taylor}{\rm)}, for each odd prime power $ q=p^f>3 $, there exists two complementary\,\footnote{~Here, `complementary' means that their corresponding designs (also called \textit{two-graphs}) are complementary, not that the graphs themselves are complementary (refer to \cite[\S\,1.5]{DRG-book}).} distance-transitive Taylor graphs $ {\rm T_1}(q) $ and $ {\rm T_2}(q) $ of unitary type, defined as follows.\vs

Let $ H $ be a Hermitian form on the projective plane $ \PG(2,q^2) $. Let $ U = \{u\,|\,H(u,u)=0\} $. A triple $ \{\alpha,\beta,\gamma\} $ on $ U $ is called a \textit{coherent} triple if the value $ H(\alpha,\beta)H(\beta,\gamma)H(\gamma,\alpha)\in\mathbb{F}_{q^2} $ is a non-square when $ q\equiv 1\,(\mbox{\rm mod\,} 4) $ or a square when $ q\equiv 3\,(\mbox{\rm mod\,} 4) $. Now, take $ z\in U $ to be a fixed point. The graph $ {\rm T_1}(q) $ is defined with vertices being the symbols $ u^+, u ^-\, (u\in U) $, whose edges are the pair $ u^\sigma w^\tau $, where $ u\ne w $ and either $ \sigma=\tau $, $ \{u,w,z\} $ coherent or $ \sigma\ne\tau $, $ \{u,w,z\} $ not coherent. The graph $ {\rm T_2}(q) $ is defined in a similar manner, where the roles of coherent and incoherent triples are reversed.

Also by \cite{DRG-book} or \cite{Taylor}, the graphs $ {\rm T_1}(q) $ and $ {\rm T_2}(q) $
are with parameters \[v=2(q^3+1),\  k=q^3,\  \mu_1=(q+1)(q^2-1)/2,\  \mu_2=(q-1)(q^2+1)/2,\] and the (full) automorphism group $ \Aut\,{\rm T_1}(q)=\Aut\,{\rm T_2}(q)=2\times\PGaU(3,q^2) $. Note that, $ {\rm T_1}(q) $ {\rm(}resp.\,$ {\rm T_2}(q) ${\rm)} is geodesic-transitive only if the number $ \mathcal{L}_3=vk\mu_1 $ {\rm(}resp.\,$ vk\mu_2 ${\rm)} divides the order $ |2\times\PGaU(3,q^2)|= q^3(q^3+1)(q^2-1)\cdot 4f $, i.e., \[(q+1)\,|\,4f, \mbox{\ or\ respectively,\ } (q^2+1)\,|\, (q+1)\cdot 4f.\] However, this does not hold, because by simple calculation we have{\rm\,:} for any odd prime power $ q=p^f>3 $,  $ (q+1)>4f $ and $  (q^2+1)>(q+1)\cdot 4f $. Hence, the family of graphs $ {\rm T_1}(q) $ {\rm(}and also $ {\rm T_2}(q) ${\rm)} are not geodesic-transitive.\vs
\end{exam}

\begin{exam}\label{exam-spora}
{\rm (1)} The incidence graph of a square $2$-$(w,k,\mu) $-design is introduced in \cite[\S\,1.6]{DRG-book}, which is distance-regular with diameter $ 3 $ and intersection array $ \{k,k-1,k-\mu,\,;1,\mu,k\} $, where $w=\frac{1}{2}v=1+k(k-1)/\mu$. According to \cite[\S\,7.6.2 (iii), pp.\,226]{DRG-book} {\rm(}see also {\rm Kantor} \cite{symD}{\rm)}, there exists a $ 2 $-$ (176,50,14) $-design $\mathcal{D}$ with $\Aut\,\mathcal{D}={\sf HS}$, whose incidence graph $ \Ga_1 $ is distance-transitive with $ v=352 $, intersection array $ \{50,49,36\,;1,14,50\} $, and $ \Aut\,\Ga_1 $ having $ \Aut\,\mathcal{D} $ as a index $ 2 $ subgroup (containing a duality on $\mathcal{D}$). Since $ \mathcal{L}_3=352\cdot 50\cdot49\cdot 36=2^8\cdot3^2\cdot5^2\cdot7^2\cdot11$, which does not divide $ |\Aut\,\Ga_1|= 2\cdot|{\sf HS}|=2^{10}\cdot 3^2\cdot 5^3\cdot 7\cdot 11 $, the graph $ \Ga_1 $ is not geodesic-transitive.\vs
    
{\rm (2)} The following graph is introduced in \cite[Thm.\,11.3.7, pp.\,364]{DRG-book}. Let $ {\HF}(3,2) $ be the Hermitian forms graph over $ \mathbb{F}_4 $, and $ \gamma $ be a fixed vertex. Define $ \Ga_2 $ to be the subgraph induced by $ {\HF}(3,2) $ on the antipodal points of $ \gamma $. Then $ \Ga_2 $ is distance-transitive with diameter $ \d=4 $, $ v=280 $, intersection array $ \{9,8,6,3\,;1,1,3,8\} $, and $ \Aut\,\Ga_2=\PGaL(3,4).2 $. Since $ \mathcal{L}_4=280\cdot 9\cdot 8\cdot 6\cdot 3$, which does not divide $ |\Aut\,\Ga_2|= 280\cdot 2^5\cdot 3^3 $, the graph $ \Ga_2 $ is not geodesic-transitive.\vs
    
{\rm (3)} Let $ \Ga_3 $ be the Patterson graph for Suzuki group {\sf Suz}, which is introduced in \cite[Thm.\,13.7.1, pp.\,410]{DRG-book}. Then $ \Ga_3 $ is distance-transitive with diameter $ \d=4 $, $ v=22880 $, intersection array $ \{280,243,144,10\,;1,8,90,280\} $, and $ \Aut\,\Ga_3={\sf Suz}.2 $. Since $ \mathcal{L}_4=22880\cdot280\cdot243\cdot144\cdot10=2^{13}\cdot3^7\cdot5^3\cdot7\cdot11\cdot13$, which does not divide $ |\Aut\,\Ga_3|= 2^{14}\cdot3^7\cdot5^2\cdot7\cdot11\cdot13 $, the graph $ \Ga_3 $ is not geodesic-transitive.\vs
    
{\rm (4)} The following graph is introduced in \cite[\S\,11.3.F, pp.\,362-363]{DRG-book} {\rm (}see also \cite[\S\,5.7.4]{ABd3}{\rm )}. Let $ T $ be the coset graph $ [2^{10}.\M_{22}.2] $ of the truncated binary Golay code, and let $ \Ga_4 $ be its bipartite double. Then $ \Ga_4 $ is distance-transitive with diameter $\d=7 $, $ v=2^{11} $, intersection array $\{22,21,20,16,6,2,1\,;1,2,6,16,20,21,22\}$, and $ \Aut\,\Ga_4=(2^{10}.\M_{22}.2)\times 2 $. Since $ \mathcal{L}_7=2^{11}\cdot22\cdot21\cdot20\cdot16\cdot6\cdot2=2^{20}\cdot3^2\cdot5\cdot7\cdot11$, which does not divide $ |\Aut\,\Ga_4|= 2^{19}\cdot3^2\cdot5\cdot7\cdot11$, the graph $ \Ga_4 $ is not geodesic-transitive.\vs\vs
\end{exam}

\section{Graphs in Table \ref{tab-main}}\label{sec-gtg} 

In this section we focus on graphs listed in Table \ref{tab-main}. The geodesics of these graphs all have a clear (often geometric) structure. Consequently, all these graphs are geodesic-transitive, that is, Theorem \ref{thm-gtg} holds. We treat the cycles $ {\rm C}_k $ as an obvious case and treat the others in Sections \ref{sec-set}\,-\,\ref{sec-gp}. For convenience, denote by $\mathcal{L}_{XY}$ the set of geodesics from $ X $ to $ Y $, where $ X,Y $ are two vertices.\vs

Most of these graphs have been studied in detail in \cite[\S\,6, \S\,9]{DRG-book}, where they are referred to as \textit{graphs with classical parameters}, \textit{partition graphs}, and \textit{regular near polygons}. We list in Table \ref{tab-ci} the intersection numbers for some of them which will be used later. Here, $ [x] $ denotes the floor function of $ x $, and the number $ [\stack{n}{m}]_q $ (or simply $ [\stack{n}{m}] $ if there is no ambiguity) denotes the $ q $-ary Gaussian binomial coefficient, i.e., $[\stack{n}{m}]_q=\frac{(q^n-1)\,\cdots\,(q^{n-m+1}-1)}{(q^m-1)\,\cdots\,(q-1)}$.\vs

\begin{table}[htb]
\centering
\caption{Intersection number $ c_i $}\label{tab-ci}
\scalebox{0.85}{
\begin{tabular}{c|c|c|c|c|c|c|c}
	\toprule[1pt]
	$ \Ga $ & $ \J(n,k) $ & $ \mathcal{O}(k)$ & $ {\G}(n,k,q)
	 $ & $ \mathcal{G}(k,q) $ & $ \DD(W)$ & $ \H(k,m) $ & $\BF(m,k,q)$\\
	\midrule[0.5pt]
	
	$ \d_{\Ga} $ & $ k $ & $ 2k-1 $ & $ k $ & $ 2k-1 $ & $ \o $ & $ k $ & $ \min\,\{m,k\} $\\
	\midrule[0.5pt]
	
	$ c_i $ & $ i^2 $ & $ [\frac{1}{2}(i+1)] $ & $ [\stack{i}{1}]^2 $ & $ [\stack{[\frac{1}{2}(i+1)]}{1}] $ & $[\stack{i}{1}]$ & $ i $ & $ q^{i-1}[\stack{i}{1}] $\\
	\bottomrule[1pt]
\end{tabular}}
\end{table}

The following two lemmas are elementary but useful.

\begin{lemma}\label{lem31}
Let $ \Ga $ be a $ G $-distance-transitive graph with intersection array $ \{b_0,b_1\cdots,b_{\d-1}\,;c_1,c_2,\cdots,c_{\d}\} $. Let $ X, Y $ be a pair of antipodal vertices at distance $ \d $. Then $ \Ga $ is $ G $-geodesic-transitive if and only if the stabilizer of the pair $ (X,Y) $ in $ G $ acts transitively on the set $\mathcal{L}_{XY} $. In particular, $ |\mathcal{L}_{XY}|= c_1c_2\cdots c_\d $.
\end{lemma}

\begin{proof}
The `only if' part is trivial. For the `if part', $ G $ is transitive on the pairs of antipodal vertices, while the stabilizer of a fixed pair $ (X,Y) $ in $ G $ acts transitively on $\mathcal{L}_{XY}$, so $ G $ is transitive on the diameters of $ \Ga $. The definition of numbers $ b_i $ implies that every geodesic of length $ i<\d $ extends at its tail end to a diameter. Then $ G $-diameter-transitivity implies $ G $-geodesic-transitivity. At last, it follows directly from the definition of $ c_i $ that $ |\mathcal{L}_{XY}|= c_1c_2\cdots c_\d $.
\end{proof}

\begin{lemma}\label{lem32}
The following statements hold{\rm\,:} {\rm(1)} a halved graph of a bipartite geodesic-transitive graph is geodesic-transitive{\rm\,;} {\rm(2)} a folded graph of an antipodal geodesic-transitive graph is geodesic-transitive. 
\end{lemma}

\begin{proof}
Recall the halved and the folded graphs introduced in Section \ref{sec-pri}. \vs

(1) Let $ \Ga $ be a bipartite $ G $-geodesic-transitive graph, and let $ \Ga_{0} $ be a halved graph of $ \Ga $ on vertex set $ V_{0} $. By restricting those elements of $ G $ that fix $ V_{0} $ to $ V_{0} $ itself, the group $ G $ induces an automorphism group $ G_0\leqslant\Aut\,\Ga_{0} $. Note that, any $ r $-geodesic $ L_0 $ in $ \Ga_{0} $ can be embedded into a $ 2r $-geodesic $ L $ in $ \Ga $ as follows\,:  $$ L_0:(X_{0},X_{1},\ldots,X_{r}) \rightarrow L:(X_{0},Y_{1},X_{1},Y_{2},\ldots,Y_{r},X_{r}),\  X_{i}\in V_{0},\  Y_{i}\notin V_{0}.$$ Since $ \Ga $ is $ G $-geodesic-transitive, it follows that $ \Ga_0 $ is $ G_0 $-geodesic-transitive.\vs
	
(2) Let $ \Ga $ be an antipodal $ G $-geodesic-transitive graph with diameter $ \d $, and let $ \overline{\Ga} $ be the folded graph. Then $ G $ naturally induces an automorphism group $ \overline{G}\leqslant\Aut\overline{\Ga} $. If $ \d=2 $ or $ 3 $, then $ \overline\Ga $ is a complete graph so it is geodesic-transitive. Suppose then $ \d>3 $, $ \Ga $ is an antipodal cover of $ \overline\Ga $. For any vertex $ v $ of $ \Ga $, the set of neighbors $ \Ga(v) $ corresponds to $ \overline{\Ga}(\bar{v}) $ one to one. Thus, any geodesic $ \overline{L}:(\overline{X}_{1},\ldots,\overline{X}_{\ell}) $ in $ \overline{\Ga} $ has a preimage $ L:(X_{1},\ldots,X_{\ell}) $ which forms a geodesic in $ \Ga $. Since $ \Ga $ is $ G $-geodesic-transitive, it follows that $\overline{\Ga}$ is $ \overline{G} $-geodesic-transitive. 
\end{proof}

\subsection{\it Johnson family}\label{sec-set}{~}\vs

In this part, let $ \Delta $ be a fixed set of size $ n $. Denote by $ \binom{\Delta}{i} $ the set of all $ i $-subsets of $ \Delta $, and by $ G_\Delta $ the symmetric group $ \Sym(\Delta) $. The \textit{Johnson graph} $ {\J}(n,k)$, $ 1\leqslant k\leqslant $ $[\frac{n}{2}]$, is defined on $ \binom{\Delta}{k}$, where $X,Y$ are adjacent whenever $ |X\cap Y|=k-1$. It has $ G_\Delta $ as an distance-transitive automorphism group (acting in a natural way). The distance is $\d(X,Y)=k-|X\cap Y|$. Further, it is known to be geodesic-transitive (\cite[Prop.\,2.1]{GTG}). We give a short, non-inductive proof\,:\vs

{\it Let $ X,Y $ be antipodal vertices at distance $ k $ with $ X\cap Y =\emptyset $, and $ \mathcal{F}_{XY} $ be the set of pairs $(F_X,F_Y) $ of subset-chains as
\[F_X:X\supset X_{k-1}\supset X_{k-2}\supset\cdots\supset X_{1}\mbox{\ and\ } F_Y: Y_{1}\subset Y_{2}\subset\cdots\subset Y_{k-1}\subset Y,\] where $ |X_i|=| Y_i|=i $. It is easy to check that the map $ f $ defined as follows is an injection from $\mathcal{F}_{XY}$ to $ \mathcal{L}_{XY}${\rm\,:}
\[f:\, (F_X,F_Y)\ \rightarrow\ (X,\,X_{k-1}\cup Y_{1},\,X_{k-2}\cup Y_{2},\,\cdots,\,X_{1}\cup Y_{k-1},\,Y).\] By {\rm Table \ref{tab-ci}}, $ |\mathcal{L}_{XY}|=c_1c_2\ldots c_{k}=1^22^2\cdots k^2=(k!)^2=|\mathcal{F}_{XY}|$, so $ f $ is a bijection. \vs

The automorphism group $ G_\Delta $ contains a subgroup $ S=\Sym(X)\times\Sym(Y)\times 1|_{\Delta\setminus(X\cup Y)} $ which fixes $ (X,Y) $ and induces an action on both $\mathcal{F}_{XY}$ and $\mathcal{L}_{XY}$. By definition, $ f $ provides a permutation equivalence between these two actions. Clearly, $ S $ is transitive on $\mathcal{F}_{XY}$, thus, transitive on $\mathcal{L}_{XY}$, too. By {\rm Lemma \ref{lem31}}, the graph $ {J}(n,k)$ is geodesic-transitive.\qed}\vs\vs

The \textit{doubled Odd graph} $ \mathcal{O}(k) $, which have $ J(2k-1,k-1) $ as a halved graph, is defined on $ \binom{\Delta}{k-1}\cup\binom{\Delta}{k} $ with $ |\Delta|=2k-1\geqslant 5 $, where $ X,Y $ are adjacent whenever $ X\subset Y\ \mbox{or}\ Y\subset X$. It is distance-transitive with the distance $ \d(X,Y)=|X|+|Y|-2|X\cap Y| $, and $ \Aut\,\mathcal{O}(k)=G_{\Delta}\times \Z_2 $. Similarly, it is geodesic-transitive\,: \vs

{\it Let $ X,Y $ be antipodal vertices at distance $ 2k-1 $ with $ |X|=k-1 $, $|Y|=k$, $ X\cap Y =\emptyset $, and let $ \mathcal{F}_{XY} $ be the set of pairs $(F_X,F_Y) $ of subset-chains as \[F_X:X\supset X_{k-2}\supset X_{k-3}\supset\cdots\supset X_{1} \mbox{\ and\ } F_Y: Y_{1}\subset Y_{2}\subset\cdots\subset Y_{k-1}\subset Y,\] where $ |X_i|=|Y_{i}|=i $. By {\rm Table \ref{tab-ci}}, $|\mathcal{L}_{XY}|=c_1c_2\cdots c_{2k-1}=1^2\,2^2\cdots(k-1)^2k =|\mathcal{F}_{XY}|$. There exists a bijection $ f:\mathcal{F}_{XY}\rightarrow\mathcal{L}_{XY}$ which maps each $ (F_X,F_Y) $ to \[(X,\,X\cup Y_{1},\,X_{k-2}\cup Y_{1},\,X_{k-2}\cup Y_{2},\,X_{k-3}\cup Y_{2},\,\cdots,\,X_{1}\cup Y_{k-1},\, Y_{k-1},\,Y).\] The automorphism group $ S\leqslant G_\Delta$ with $S=\Sym(X)\times\Sym(Y)\times 1|_{\Delta\setminus(X\cup Y)} $ fixes the pair $ X,Y $. It is transitive on $\mathcal{F}_{XY}$, thus, equivalently, transitive on $\mathcal{L}_{XY}$.\qed}\vs\vs

The Johnson graph $ \J(2k,k) $ is an antipodal $ 2 $-cover of the \textit{folded Johnson graph} $ \bar \J(2k,k) $, defined on partitions of $ \Delta $ into two $ k $-sets, two partitions being adjacent whenever their common refinement is with four parts of sizes $ 1 $, $ k-1 $, $ 1 $, $ k-1 $. The graph $ \mathcal{O}(k) $ is an antipodal $ 2 $-cover of the \textit{Odd graph} $ \O(k) $, defined on $ \binom{\Delta}{k-1} $ with $ |\Delta|=2k-1 $, two vertices being adjacent whenever they are disjoint. By Lemma \ref{lem32}, $ \bar \J(2k,k) $ and $ \O(k) $ are both geodesic-transitive.\vs

In conclusion,  we have the following result.

\begin{pro}
Each of the following graphs is geodesic-transitive{\rm\,:} $ \J(n,k) $, $\bar{\J}(2k,k)$, $\mathcal{O}(k)$ and $\O(k)$.
\end{pro}

\subsection{\it Graph on projective space}\label{sec-prospace}{~}\vskip0.03in

In this part, let $ V $ be a vector space of dimension $ n $ over $ \mathbb{F}_q$, and $ \PG(V) $ be the associated projective space. Denote by $ [\stack{V}{i}] $ the set of all $ i $-subspaces of $ V $. \vs

The \textit{Grassmann graph} ${\G}(n,k,q)$ (or briefly, $ \G_V(k) $), $ 1\leqslant k\leqslant $ $[\frac{n}{2}]$, is defined on $ [\stack{V}{k}]$, where $X,Y$ are adjacent whenever $ dim\ X\cap Y=k-1$. The \textit{doubled Grassmann graph} $ \mathcal{G}(k,q) $ is defined on $ [\stack{V}{k-1}]\cup[\stack{V}{k}] $ with $ dim\ V=2k-1 $, where $ X,Y $ are adjacent whenever $ X\subset Y\ \mbox{or}\ Y\subset X$. These graphs are geodesic-transitive (\cite[Thm.\,2.6 \&\,2.9]{Hua}).\vs

The (doubled) Grassmann graphs graphs are viewed as the `$ q $-analogues' of the Johnson family. The geodesics are described in a similar manner. We only treat the Grassmann $ \G_V(k)$ which will be used later. The distance is $\d(X,Y)=k-dim\ X\cap Y$. Let $ X, Y $ be antipodal vertices at distance $ k $ with $ dim\ X\cap Y=0 $, and $ \mathcal{F}_{XY} $ be the set of pairs $(F_X,F_Y) $ of (maximal) flags\,:
\[F_X:X> X_{k-1}> X_{k-2}>\cdots> X_{1}\mbox{\ and\ } F_Y: Y_{1}< Y_{2}<\cdots< Y_{k-1}< Y,\] where $ dim\  X_i=\, dim\ Y_i=i $. By Table \ref{tab-ci}, $|\mathcal{L}_{XY}|=c_1c_2\cdots c_k=([\stack{1}{1}])^2([\stack{2}{1}])^2\cdots$ $([\stack{k}{1}])^2=([\stack{1}{1}][\stack{2}{1}]\cdots[\stack{k}{1}])^2=|\mathcal{F}_{XY}|$. There exists a bijection $f:\mathcal{F}_{XY}\rightarrow\mathcal{L}_{XY}$ as
\[f:\, (F_X,F_Y)\ \rightarrow\ (X,\,X_{k-1}+ Y_{1},\,X_{k-2}+ Y_{2},\,\cdots,\,X_{1}+Y_{k-1},\,Y).\]

\noindent \textbf{Remark.} We correct a minor error in book \cite{DRG-book}. The \textit{bipartite double} of a graph is introduced in \cite[\S\,1.11]{DRG-book}. The doubled Odd graph $\mathcal{O}(k)$ are in fact the bipartite double of the Odd graph $ \O(k) $\,; similarly, the doubled Grassmann graph $\mathcal{G}(k,q)$ is considered the bipartite double of the distance-$ (k-1) $-graph of Grassmann graph $ \G(2k-1,k-1,q) $ in \cite[pp.272]{DRG-book}. However, a simple calculation shows that these two graphs have different valencies, so this is NOT true.\vs\vs 

There are two more families of graphs arisen from $ \PG(V) $. Let $ \D $ be the symmetric design of the points and hyperplanes in $ \PG(V) $. The incidence graph $ {\rm I}({\sf D}) $ of $ \D $ (or its complementary design) is geodesic-transitive with diameter $ 3 $ (\cite[Thm.\,3.3]{Hua}). These graphs are $ B $-doubles of complete graphs.\vs

Let $ (w,H) $ be an incident point-hyperplane pair in $ \D $. Denote by $ \Delta^{(1)} $ the set of all points not in $ H $ and by $ \Delta^{(2)} $ the set of all hyperplanes not containing $ w $. The \textit{opposites} $ \D^{op\,(w,H)} $ to $ (w,H) $ in $ \D $ is defined to be the incidence geometry \[ \D^{op\,(w,H)}=(\Delta^{(1)},\Delta^{(2)}, \mbox {incidence\ of } \D).\] Let $ \Ga={\rm I}(\D^{op\,(w,H)}) $ be the incidence graph of $ \D^{op\,(w,H)} $. According to \cite[Main\,Thm.\,(1)]{AKnn} (where the dimension is assumed to be $ n+1 $), $ \Ga $ is a distance-transitive antipodal $ q $-cover of $ K_{q^{n-2},q^{n-2}} $, with intersection numbers \[c_1=1,\  c_2=q^{n-3},\  c_3=q^{n-2}-1,\  c_4=q^{n-2}.\] It is easy to see that the parabolic subgroup of $ \PGL(V) $ that stabilizes the pair $ (w,H) $ induces an automorphism group $ T\leqslant\Aut\,\Ga $.

\begin{pro}
The graph $\Ga={\rm I}(\D^{op\,(w,H)}) $ defined above is geodesic-transitive.
\end{pro}

\begin{proof}
By \cite[Cons.\,1.1]{AKnn}, each antipodal block in $ \Delta^{(1)} $ is the set of points (apart from $ w $) of a fixed $ 2 $-dimensional subspace on $ w $. To clarify the situation, fix a basis $ \{v_1,v_2\ldots,v_n\} $ for $ V $ and a pair of antipodal points $ X,Y $ in $\Delta^{(1)}$ such that \[w=\l v_1\r,\ H=\l v_1,v_2\ldots,v_{n-1}\r,\ X=\l v_n\r, \ Y=\l v_n+v_1\r.\] 
	
Let $\mathcal{F}_{XY}$ be the set of pairs $ (U,\l u\r) $, where $ U $ is a complementary subspace of $ X+Y $ in $ V $ that is not contained in $ H $, and $ \l u\r\in\Delta^{(1)}$ is point lying in $ U $. We calculate the numbers of such pairs. The subspace $ U $ is of form \[U=\l v_2+m_2,\,v_3+m_3,\,\ldots,\,v_{n-1}+m_{n-1}\r,\] where the vectors $ m_2,\ldots,m_{n-1}$ lie in $ X+Y $, and (since $ U\nleq H $) at least one of them does not lie in $\l v_1\r$. Thus, there are $ q^{2(n-2)}-q^{(n-2)} $ choices for $ U $. Further, each $ U $ contains $ (q^{n-2}-1)/(q-1) $ $ 1 $-spaces while $ (q^{n-3}-1)/(q-1) $ of them are contained in $ U\cap H $. Then there are $ q^{n-3} $ choices for $ \l u\r $. Hence, we have  \[|\mathcal{L}_{XY}|=c_1c_2c_3c_4=q^{n-3}\cdot q^{n-2}(q^{n-2}-1)=|\mathcal{F}_{XY}|.\]
	
For any pair $ (U,\l u\r)\in\mathcal{F}_{XY} $, both $ X+U $ and $ Y+U $ are hyperplanes not containing $ w $. Then there exists a map $ f $ from $\mathcal{F}_{XY}$ to $\mathcal{L}_{XY}$ defined as follows{\rm\,:} \[f:\, (U,\l u\r)\ \rightarrow\ (X,\,X+U,\,\l u\r,\,Y+U,\, Y).\] As $ U$ is determined by $U=(X+U)\cap (Y+U) $, $ f $ is an injection (so, a bijection).\vs

Recall $ T\leqslant\Aut\,\Ga $, the group of automorphism induced by the stabilizer of $ (w,H) $ in $ \PGL(V) $. Let $ S\leqslant T $ be the subgroup that fixes $ (X,Y) $. Then $ S $ induces an action on both $\mathcal{F}_{XY}$ and $\mathcal{L}_{XY}$. By definition, $ f $ provides a permutation equivalence between these two actions. Then, by Lemma \ref{lem31}, we only need to prove that $ S $ is transitive on $\mathcal{F}_{XY}$, that is, to give some $ s\in S $ that maps a fixed pair in $\mathcal{F}$, say $ (U_0,\l u_0\r) $ with \[U_0=\l v_2+v_n,\, v_3,v_4\ldots, v_{n-1}\r,\ \l u_0\r=\l v_2+v_n\r,\] to any other one $ (U,\l u\r) $. Write $ U $ in the above form, with $ m_{i} $ not lying in $ \l v_1\r $, so $m_{i}=\alpha v_1+\beta v_n$ where $ \beta\ne0 $. Via adding appropriate scalar multiples of the vector $ v_i+m_i $ to each $ v_j+m_j\  (j\ne i) $ to eliminate the $ v_n $ term, rewrite $ U $ as \[U=\l h_2,\ldots, h_{i-1},\,v_i+\alpha v_1+\beta v_n,\,h_{i+1},\ldots,h_{n-1}\r, \mbox{\ with\ each\ } h_j\in H .\] Then $ \l u\r\leqslant U$ with $ \l u\r\nleq H $ (as $ \l u\r\in\Delta^{(1)} $) can be written as \[\l u\r=\l v_i+\alpha v_1+\beta v_n+\gamma\r, \mbox{\ with\ }\gamma\in\l h_2,\ldots,h_{i-1},h_{i+1},\ldots,h_{n-1}\r\leqslant H.\] 
Define $ g\in\GL(V) $ as follows and let $\bar{g}$ be its projective image\,:\vs
$v_1^g=\beta v_1$,\ \ $v_n^g=\beta v_n$, \vs
$ v_2^g=v_i+\alpha v_1+\gamma$,\ \ $  v_i^g=h_2$, and $ v_k^g=h_k$ for $ k\notin\{1,n,2,i\} $. \vs
\noindent It is straightforward to verify\,: (1) $\bar{g}$ stabilizes both $ (w,H) $ and $ (X,Y) $, so it induces some $s\in S $\,; (2) $\bar{g}$ maps $ (U_0,\l u_0\r) $ to $ (U,\l u\r) $. Then we are done.
\end{proof}

\subsection{\it Dual polar graph}\label{sec-polarspace}{~}\vskip0.03in

In this part, let $ V $ be a vector space over $ \mathbb{F}_q $ and $ B $ be a non-degenerate sesquilinear reflexive form on $ V $ such that $ W=(V,B) $ is one of the spaces listed in Table \ref{tab-pspace}. Let $ G_{W} $ be the associated group of $ W $.\vs

For $ U\leqslant V $, let $ U^\perp=\{x\in V\ |\, B\,(x,U)=0\} $. A subspace $ X $ of $ V $ is called \textit{singular}\,\footnote{In standard terminology it is called \textit{totally isotropic}, or, in certain cases, \textit{totally singular}.} if $ X\leqslant X^\perp $. By the famous Witt's Lemma, every isometry between subspaces of $ W $ extends to an isometry of $ W $ itself. It follows that all maximal singular subspaces have the same dimension $ \o $, which is referred to as the \textit{Witt index}. We assumed that $ \o\geqslant 2 $. A \textit{polar frame} of $ W $ is a set of  vectors $ \{x_1,\ldots,x_{\o}\,;\, y_1,\ldots,y_{\o}\} $ which are linearly independent and satisfies that\,: \[B(x_i,x_j)=B(y_i,y_j)=0\mbox{\ and\ } B(x_i,y_j)=\delta_{ij}.\] More information on these spaces can be found in many texts, such as \cite{CG-book}.\vs
\begin{table}[htb]
\centering
\caption{Formed spaces $ W $ of Witt index $ \o $}\label{tab-pspace}
\scalebox{0.75}{
\begin{tabular}{c|c|c|c|c|c|c}
	\toprule[1pt]
	$ W $ & $ Sp(2\o,q) $ & $\Omega(2\o+1,q)$ & $\Omega^+(2\o,q)$ & $\Omega^-(2\o+2,q)$ & $U(2\o+1,q^{\frac{1}{2}})$ & $U(2\o,q^{\frac{1}{2}})$\\
	\midrule[0.5pt]
	
	$G_W$ & P$ \Sigma $Sp$ (2\o,q) $ & P$ \Gamma $O$ (2\o+1,q) $ & P$ \Gamma $O$^{+} (2\o,q) $ & P$ \Gamma $O$^{-} (2\o+2,q) $ & P$ \Gamma $U$ (2\o+1,q^{\frac{1}{2}})$ & P$ \Gamma $U$ (2\o,q^{\frac{1}{2}})$ \\
	\bottomrule[1pt]
\end{tabular}}
\end{table}\vs

For $ 1\leqslant i\leqslant\o $, denote by $ [\stack{W}{i}]_0 $ the set of isotropic $ i $-subspaces of $ W $. The \textit{dual polar graph} $ \DD(W) $ is defined on $ [\stack{W}{\o}]_0 $, where $ X,Y $ are adjacent whenever $ dim\ X\cap Y=\o-1 $. It has $ G_W $ as a distance-transitive automorphism group. The distance $\d(X,Y)=\o-dim\ X\cap Y$ is as in Grassmann graphs.

\begin{pro}\label{prop-dualpolar}
The dual polar graph $ \DD(W) $ is geodesic-transitive.
\end{pro}

\begin{proof}
Let $ X,Y $ be antipodal vertices at distance $ \o $ with $ \dim\ X\cap Y=0  $, and $ \mathcal{F}_{XY} $ be the set of pairs $ (F_{X},F_{Y}) $ of maximal flags as
\[F_X : X=X_\o> X_{\o-1} >\cdots> X_{1}\mbox{\ and\ } F_Y: Y_{1}< Y_{2}<\cdots< Y_\o=Y,\] where each $ Y_i $ with $ dim\ Y_i=i $ is taken arbitrarily, while each $ X_i $ is correspondingly taken as $ X_i=Y_{\o-i}^\perp\cap X $. Since $ Y_{\o-1}^\perp<Y_{\o-2}^\perp<\cdots<Y_1^\perp $, it follows that $ X_1\leqslant X_2\leqslant\cdots\leqslant X_{\o-1} $. Further, $dim\ Y_{\o-i}^\perp=dim\ V-(\o-i) $, so $ dim\ X_i\geqslant i $\,; $ dim\ X_i+Y_{\o-i}\leqslant\o $ (the Witt index), so $ dim\ X_i\leqslant i $. Thus, $ dim\ X_i= i $.\vs

It is easy to check that the following map $ f $ is an injection from $\mathcal{F}_{XY}$ to $ \mathcal{L}_{XY}${\rm\,:}
\[f:\, (F_X,F_Y)\ \rightarrow\ (X,\,X_{\o-1} + Y_{1},\, X_{\o-2}+Y_{2},\,\cdots,\, X_{1}+Y_{\o-1},\,Y).\]
By {\rm Table \ref{tab-ci}}, $ |\mathcal{L}_{XY}|=c_1c_2\ldots c_{\o}=[\stack{1}{1}][\stack{2}{1}]\cdots [\stack{\o}{1}]=|\mathcal{F}_{XY}|$, so $ f $ is a bijection. \vs

Recall that $ \DD(W) $ has $ G_W $ as an automorphism group. Let $ S\leqslant G_W $ be the subgroup that fixes $ (X,Y) $. Then $ S $ induces an action on both $\mathcal{F}_{XY}$ and $\mathcal{L}_{XY}$. By definition, $ f $ provides a permutation equivalence between these two actions. Then, by Lemma \ref{lem31}, it is sufficient to prove that $ S $ is transitive on $\mathcal{F}_{XY}$.\vs

After a typical process of normalization (i)\,-\,(iii) as follows, we obtain a polar frame $ \{x_1,\ldots,x_{\o}\,;\, y_1,\ldots,y_{\o}\} $ of $ W $ which generates $ (F_X, F_Y) $ as \[X_i=\l x_{\o},x_{\o-1},\ldots,x_{\o-i+1}\r \mbox{\ and \ } Y_i=\l y_1,y_2,\ldots, y_i\r.\ \ \ (\ast)\]
	
(i) Take vectors $ x_\o $, $y_\o$ such that $ X_1=\l x_\o\r $, $ Y/Y_{\o-1}=\l y_\o\r $. As $ X_1\leqslant Y_{\o-1}^\perp $, if $ B(x_\o,y_\o)= 0 $, we get a singular subspace $ X_1+Y $ of dimension $ \o+1 $, not possible. Thus $ B(x_\o,y_\o)\ne 0 $, which can be set to $ 1 $.\vs

(ii) Take vectors $ x_{\o-1} $ such that $ X_2=\l x_{\o-1}, x_\o\r$ and $y_{\o-1}\in Y_{\o-1}$ such that $ Y/Y_{\o-2}=\l y_{\o-1}, y_\o \r$. Then $ B(x_{\o},y_{\o-1})=0 $, and we can set $ B(x_{\o-1},y_\o)=0 $ by redefining $ x_{\o-1} $ as $ x_{\o-1} $ plus an appropriate multiple of $ x_{\o} $. Since $ X_2\leqslant Y_{\o-2}^\perp $, it follows that $ B(x_{\o-1},y_{\o-1})\ne 0 $, which can be set to $ 1 $.\vs

(iii) In a similar manner, for $ i $ from $ \o-2 $ to $ 1 $, we can successively obtain $ x_i $ and $ y_i\in Y_i $ such that $ B(x_i,y_k)=B(x_k,y_i)=0 $ for $ k>i $ and $ B(x_i,y_i)=1 $.\vs\vs
	
Let $ (F_X', F_Y') $ be any other pair in $\mathcal{F}_{XY}$. Then we have another polar frame $ \{x_1',\ldots,x_{\o}', y_1',\ldots,y_{\o}'\} $ which generates $ (F_X', F_Y') $ in a manner similar to ($ \ast $). Let $ g $ be the linear transformation on $ X+Y $ which maps each $ x_i $ to $ x_i' $, $ y_i $ to $ y_i' $. Then $ g $ stabilizes the pair $ (X,Y) $ and maps $ (F_X,F_Y) $ to $ (F_X',F_Y') $. In fact, $ g $ is an isometry on the subspace $ (X+Y,B|_{X+Y}) $ of $ W $. By Witt's Lemma, $ g $ expands to some isometry on $ W $ lying in $ S $. Thus, $ S $ acts transitively on $\mathcal{F}_{XY}$.
\end{proof}
	
The dual polar graphs on $\Omega^+(2\o,q)$ are imprimitive, whose halved graphs are called the \textit{half dual polar graphs}. By Lemma \ref{lem32}, we have the following result.

\begin{corol}
A half dual polar graph is geodesic-transitive.
\end{corol}

\subsection{\it Hamming family}\label{sec-tuples}{~}\vs

In this part, let $ \Delta=\{0,1,\ldots,m-1\} $ be a fixed set of size $ m\geqslant 2 $. The \textit{Hamming graph} $ {\H}(k,m) $ is defined on the set $ \Delta^{k} $ of $ k $-tuples on $ \Delta $, where two vertices are adjacent whenever they have $ 1 $ different entries. The distance $ \d(X,Y)=i $ whenever $ X,Y $ have $ i $ different entry. It has $ S_m \wr S_k $ as the automorphism group, where each $ S_m $ acts on a corresponding entry and $ S_k $ permutes the $ k $ entries. It is geodesic-transitive (\cite[Prop.\,2.2]{GTG}). We give a short, non-inductive proof\,:\vs

{\it Let $ X=(0,0,\ldots,0) $, $ Y=(1,1,\ldots,1)$ be two antipodal vertices at distance $ k $. Let $ \Sigma=\{1,2,\ldots,k\} $ be the set of the positions of the entries and $ \mathcal{F} $ be the set of subset-chains $ F $ as \[F:\Sigma_1\subset \Sigma_{2}\subset \cdots\subset \Sigma_{k-1}\subset\Sigma_k= \Sigma,\mbox{\ where\ } |\Sigma_i|=i.\] It is easy to check that the following map $ f $ is an injection from $\mathcal{F}$ to $ \mathcal{L}_{XY}${\rm\,:} \[f:\,F\ \rightarrow\ (X,\,X_{1},\,X_{2},\,\cdots,\,X_{k-1},\,X_k=Y),\] where each $ X_i $ has the $ j $-th entry equal to $ 1 $ or $ 0 $, respectively, if $ j\in \Sigma_i $ or $ j\notin\Sigma_i $. By {\rm Table \ref{tab-ci}}, $ |\mathcal{L}_{XY}|=c_1c_2\cdots c_k=1\times 2\times\cdots \times k=k!=|\mathcal{F}|$, so $ f $ is a bijection.\vs
	
The automorphism group $ S_k $ fixes the pair $ (X,Y) $ and induces an action on both $\mathcal{F}$ and $\mathcal{L}_{XY}$. By definition, $ f $ provides a permutation equivalence between these two actions. Clearly, $ S_k $ is transitive on $\mathcal{F}$, thus, transitive on $\mathcal{L}_{XY}$, too. By {\rm Lemma \ref{lem31}}, the graph $ \H(k,m)$ is geodesic-transitive.\qed}\vs

From the \textit{$ k $-cube} $ \H(k,2) $ we obtain new graphs\,: the \textit{halved $ k $-cube} $\frac{1}{2}\H(k,2)$, the \textit{folded $ k $-cube} $ \overline {\H}(k,2) $, and (if $ k=2k_0 $ is even) the \textit{halved folded $ k $-cube} $ \frac{1}{2}\overline \H(k,2) $. The following result follows by Lemma \ref{lem32}. 

\begin{pro}
Each of the following graphs is geodesic-transitive{\rm\,:} $ \H(k,m) $, $\frac{1}{2}\H(k,2)$, $ \overline {\H}(k,2) $ and $ \frac{1}{2}\overline \H(k,2) $.
\end{pro}

\subsection{\it Forms graph}\label{sec-matirx}{~}\vs

In this part, let $ U=F^m $, $ V=F^k $ be vector spaces over $ F=\mathbb{F}_q $, $ m\leqslant k $. Let $ \mathcal{B} $ be the vector space of bilinear maps from $ U\times V $ to $ F $ (which is isomorphic to ($ U\otimes V)^\ast $). The rank $ \rk(f) $ is the codimension of each of its null spaces (in $ U $ and $ V $). The \textit{bilinear forms graph} $ {\BF}(m,k,q) $ is defined on $ \mathcal{B}$, where $ X,Y $ are adjacent whenever $ \rk(X-Y)=1 $. The distance is $ \d(X,Y)=\rk (X-Y) $. It is distance-transitive and has an automorphism group \[ F^{mk}{:}(\GL(U)\circ\GL(V)),\] where $ F^{mk} $ consists of the translations on $ U\otimes V $, and for $ g=(h,k)/(\lambda,\lambda^{-1})\in\GL(U)\circ\GL(V) $, $ X\in \mathcal{B}$, the image $ X^g $ is the map that takes the same value on $ U^{\lambda\cdot h}\times V^{\lambda^{-1}\cdot k} $ as $ X $ does on $ U\times V $.\vs

Let $\mathcal{A}$ be the vector space  (of dimension $ k(k-1)/2 $ over $ F $) of alternating forms on $ V $. The \textit{alternating forms graph} $ {\AF}(k,q) $ is defined on $ \mathcal{A} $, where $ X,Y $ are adjacent whenever $ \rk(X-Y)=2 $. The distance is$ \d(X,Y)=\frac{1}{2}\rk (X-Y) $. It is distance-transitive and has an automorphism group $ F^{n(n-1/2)}{:}\GL(V) $. \vs

In the case where $ q=r^2 $, let $ \sigma $ be the field automorphism of $ F $ of order $ 2 $, and let $\mathcal{H}$ be the vector space  (of dimension $ k^2 $ over $ \mathbb{F}_r $) of Hermitian forms on $ V $. The \textit{Hermitian forms graph} $ {\HF}(k,r) $ is defined on $ \mathcal{H} $, where $ X,Y $ are adjacent whenever $ \rk(X-Y)=1 $. The distance is $ \d(X,Y)=\rk (X-Y) $. It is distance-transitive and has an automorphism group $ \mathbb{F}_{r}^{k^2}{:}\GL(V) $.\vs

\begin{pro}
Each of the forms graphs $ {\BF}(m,k,q) $, $ {\AF}(k,q) $ and $ {\HF}\,(k,r) $ is geodesic-transitive.
\end{pro}

\begin{proof}
(1) We first consider the bilinear forms graph $ {\BF}(m,k,q) $. Take antipodal vertices $ X,Y $ at distance $ m $, which, with respect to given bases $ \eta_U $ for $ U $ and $ \eta_V $ for $ V $, are represented by matrices $M_X=(0_{m\times k}) $, $ M_Y=(I_{m}\  0_{m,k-m}) $. Let $\mathcal{F}$ be the set of pairs $ (F,\bar{F}) $ of (maximal) flags as \[F: U>U_{m-1}> U_{m-2}>\cdots>U_{1}, \mbox{\ and\ } \bar{F} : \bar{U}_{m-1}<\bar{U}_{m-2}<\cdots <\bar{U}_{1}<U,\] where $ dim\ U_i= dim\ \bar{U}_{m-i}=i $ and $ \bar{U}_i $ is a complementary subspace of $ U_i $ in $ U $.\vs

It is easy to check that the following map $ f $ is an injection from $\mathcal{F}$ to $ \mathcal{L}_{XY}${\rm\,:} \[f:\ (F,\bar{F})\ \rightarrow\ (X,Y_{m-1},Y_{m-2},\ldots,Y_{1},Y),\] where each $ Y_i $ has $ U_i $ as the null space in $ V $ and $ \bar{U}_i $ as the subspace on which $ Y_i=Y $. The number of choices for $ F $ is $ [\stack{m}{1}][\stack{m-1}{1}]\cdots[\stack{2}{1}] $\,; meanwhile, for a fixed $ F $, the number of choices for $\bar{F}$ is $ q^{m-1}q^{m-2}\cdots q^1 $. By table \ref{tab-ci}, $|\mathcal{L}_{XY}|=c_1c_2\cdots c_{m}=(q^{0}[\stack{1}{1}])(q^{1}[\stack{2}{1}])\cdots (q^{m-1}[\stack{m}{1}])=|\mathcal{F}|$, so $ f $ is a bijection.\vs

Recall that $ {\BF}(m,k,q) $ has $ G= \GL(U)\circ\GL(V) $ as an automorphism group. Let $ S\leqslant G $ be the subgroup that fixes $ (X,Y) $, which induces an action on both $\mathcal{F}$ and $\mathcal{L}_{XY}$. By definition, $ f $ provides a permutation equivalence between these two actions. It is sufficient to prove that $ S $ is transitive on $\mathcal{F}$ by Lemma \ref{lem31}.\vs
	
For any $ A\in\GL(U) $, there is a corresponding $B=
\begin{pmatrix} 
	(A^{\sf T})^{-1} & 0_{m,k-m}\\ 
	0_{k-m,m} &  I_{k-m} & 
\end{pmatrix}\in\GL(V) $ such that $ X,Y $, with respect to bases $ A\eta_U $, $ B\eta_V $, are respectively represented by matrices \[A^{\sf T}M_X B=0=M_X,\mbox{\ and\ } A^{\sf T}M_Y B=A^{\sf T}(I_{m}\  0_{m,k-m})B=M_Y.\]
Thus, the image of $ (A,B) $ in $ G $ fixes the pair $ (X,Y) $. It follows that $ S $ contains all the projective transformations on  $ U $. Consequently, $ S $ is transitive on the maximal flags of $ U $, and the stabilizer of a fixed flag $ F $, as a parabolic subgroup containing the unipotent radical, is transitive on those flags $\bar{F}$ consisting of complementary subspaces of each space lying in $ F $. Therefore, $ S $ is transitive on $\mathcal{F}$.\vs\vs

(2) The proof for $\Ga= {\AF}(k,q)$ or ${\HF}\,(k,r) $ proceeds in a manner similar to $ {\BF}(m,k,q) $. For convenience, we reuse the notations $\mathcal{F},f,S, X,Y $ in this part.\vs

Let $ V=V'\oplus V_0 $ be a fixed factorization of $ V $, where, if $\Ga= {\AF}(k,q)$ and $ k $ is odd, then $ dim\ V'=k-1 $\,; otherwise, $ V'=V $. Let $ X $ be the zero-form and $ Y $ be a non-degenerate form on $ V' $ with the null space in $ V $ being $ V_0 $. Then $ X,Y $ are antipodal vertices. Recall that $ \Ga $ has $ \GL(V) $ as an automorphism group. Let $ S\leqslant \GL(V) $ be the subgroup consisting of elements which fixes $ V' $, $ V_0 $, and acts as an isometry on $ (V',Y) $. Then $ S $ fixes the pair $ (X,Y) $.\vs

For $\Ga= {\AF}(k,q) $, let $\mathcal{F}$ be the set of flags $ F $ as \[F: V'>V_{\lfloor k/2\rfloor-1}>V_{\lfloor k/2\rfloor-2}>\cdots> V_1,\] where $ dim\ V_i=2i $ and the restriction $ Y|_{V_i} $ is non-degenerate. Then each $ V_i $ leads to a factorization $ V'=V_i\oplus V_i^{\perp} $, where $ V_i^{\perp}:=\{v\in V'\,|\,Y(v,V_i)=0\} $. It is easy to check that the following map $ f $ is an injection from $\mathcal{F}$ to $ \mathcal{L}_{XY}${\rm\,:} \[f:\ F\ \rightarrow\ (X,Y_{\lfloor k/2\rfloor-1},Y_{\lfloor k/2\rfloor-2},\ldots,Y_{1},Y),\] where each $ Y_i $ has $ V_i+V_0 $ as the null space and $ V_i^{\perp} $ as the subspace on which $ Y_i=Y $. By \cite[Thm.\,9.5.6]{DRG-book}, the intersection number $ c_i=q^{2i-2}(q^{2i}-1)/(q^2-1) $ is equal to the number of non-degenerate $ 2 $-spaces contained in a given non-degenerate $ 2i $-space. Thus, $|\mathcal{L}_{XY}|=c_1c_2\cdots c_{\lfloor k/2\rfloor}=|\mathcal{F}|$, and $ f $ is a bijection. Moreover, by definition $ f $ provides a permutation equivalence between the actions of $ S $ on $\mathcal{F}$ and on $\mathcal{L}_{XY}$. Since $ S $ has the isometry group of $ (V',Y) $ as a subgroup, it follows by Witt's Lemma that $ S $ is transitive on $\mathcal{F}$, and we are done.\vs
    
For $ \Ga={\HF}\,(k,r) $, let $\mathcal{F}$ be the set of flags $ F $ as \[F: V>V_{k-1}>V_{k-2}>\cdots> V_1,\] where $dim\ V_i=i $ and the restriction $Y|_{V_i} $ is non-degenerate. The number $ c_i $ is given in \cite[Thm.\,9.5.7]{DRG-book}. In a similar manner, there exists a bijection from $ \mathcal{F} $ to $\mathcal{L}_{XY}$, and the equivalent actions of $ S $ on $ \mathcal{F} $ and $\mathcal{L}_{XY}$ are both transitive.
\end{proof}

\noindent\textbf{Remark.} A forms graph $ \Ga $ can also be defined on matrices over $ F $\,: respectively, $ {\BF}(m,k,q) $, $ {\AF}(k,q) $ or $ {\HF}(k,r) $ is defined on the \textit{rectangular} $ (m\times k) $ matrices, the \textit{skew symmetric} $ (k\times k) $-matrices (with zero diagonal) or the \textit{Hermitian} $ (k\times k) $-matrices with $ M^{\sf T}=M^{\sigma} $, where $ X,Y $ are adjacent whenever $ \rk(X-Y) $ is equal to $ 1 $, $ 2 $ or $ 1 $. From this perspective, for any geodesic of length $ \ell $, there is an automorphism that maps it to $ (0,L_{1},\ldots,L_{\ell}) $, where, if $\Ga= {\BF}(m,k,q)$ or $ {\HF}(k,r) $, then $L_{i}=\begin{pmatrix} I_{i} & 0 \\ 0 & 0 \end{pmatrix}$\,; if $ \Ga={\AF}(k,q) $, then $ L_{i}=\begin{pmatrix}\begin{smallmatrix}
0 & I_{i} & 0 \\ -I_{i} & 0 & 0\\ 0 & 0 & 0 \end{smallmatrix}\end{pmatrix}$.\vs

\subsection{\it Generalized polygon}\label{sec-gp}{~}\vskip0.03in

A \textit{generalized $ n $-gon} is a partial linear space $\Ga=(\mathcal{P},\mathcal{L},\,{\bf I})$ with point set $\mathcal{P}$, line set $\mathcal{L}$ and symmetric relation $ {\bf I} $, whose incidence graph has diameter $ n $ and girth $ 2n $ (there are several equivalent definitions, cf.~\cite[\S\,6.5]{DRG-book} and \cite[\S\,1]{GP-book}). The \textit{dual} of $ \Ga $ is the space $ \Ga^D=(\mathcal{L},\mathcal{P},\,{\bf I}) $ which interchanges the roles of $\mathcal{P}$ and $\mathcal{L}$.\vs

Here, we regard the generalized polygon $ \Ga $ as its \textit{collinearity graph} on $\mathcal{P}$ which has diameter $\d=\lfloor n/2\rfloor$, and denote by $ {\rm I}(\Ga) $ the incidence graph of $ \Ga $. A \textit{collineation} of $ {\rm I}(\Ga) $ is a graph automorphism which preserves the two biparts $\mathcal{P}$, $\mathcal{L}$. It is clear that any collineation of $ {\rm I}(\Ga) $ induces an automorphism of $ \Ga $.\vs

The finite distance-transitive generalized polygons were classified by Buekenhout and Van Maldeghem \cite{GP}. Those with diameter $ \d\geqslant 3 $ are the following generalized hexagons (\GH), octagons (\GO), dodecagons (\GD) and their duals. The associated group $ G $ is also given. All these generalized polygons exhibit a high degree of symmetry, which implies geodesic-transitivity, as detailed below.

\begin{table}[htb]
\centering
\caption{Distance-transitive generalized polygons with $ \d\geqslant 3 $}
\scalebox{0.8}{
\begin{tabular}{c|c|c|c|c|c|c}
	\toprule[1pt]
	$ (\Ga,\d) $ & $(\GH(1,q),3)$ & $ (\GO(1,q),4) $ & $(\GD(1,q),6)$ & $(\GH(q,q),3)$ & $(\GH(q,q^3),3)$  & $(\GO(q,q^2),4)$  \\
	\midrule[0.5pt]
	
	$G$ & $ {\rm PSL}(3,q) $ & $ {\rm Sp}_4(q), q=2^a $ & $ {\rm G}_2(q), q=3^a $ & $ {\rm G}_2(q) $ & $ ^3{\rm D}_4(q) $ & $ ^2{\rm F}_4(q)  $ \\
	\bottomrule[1pt]
\end{tabular}}
\end{table}

(a) Let $ \Ga\in\{\GH(1,q), \GO(1,q), \GD(1,q)\} $ be a distance-transitive generalized $ 2\d $-gon with diameter $ \d $. Then $ \Ga $ is isomorphic to $ I(X) $ where $ X $ is a self-dual generalized $\d$-gon of order $ (q,q) $\,; the dual $ \Ga^D $ is of order $ (q,1) $ and is isomorphic to the \textit{line graph} of $ I(X) $, namely, $ \Ga^D $ has the incident point-line pairs of $ X $ as vertices and two vertices are adjacent whenever they have common element. Let $ L: (V_0,V_1,\ldots,V_{\ell}) $ be a geodesic in $ \Ga^D $ of length $ \ell\leqslant\d $. Then for $ 0<i<\ell $, \[V_i=(V_{i-1}\cap V_i)\cup (V_{i+1}\cap V_i),\] and consequently, there exists an arc $ (a_0,a_1,\ldots,a_{\ell+1}) $ in $ \Ga $ of length $ \ell+1\leqslant \d+1 $ which induces $ L $ as each $ V_i=\{a_i,a_{i+1}\} $. According to \cite[Thm.\,1.1]{GTG}, the graph $ \Ga $ is $ (\d+1) $-arc transitive, which implies that both $ \Ga $ and $ \Ga^D $ are geodesic-transitive. \vs

{\rm (b)} Let $ \Ga\in\{\GH(q,q),\GH(q,q^3), \GO(q,q^2)\} $ be a distance-transitive $ n $-gon with diameter $ \frac{n}{2} $. By a basic property of generalized $ n $-gons, any two vertices in $ \Ga $ at distance $i<\frac{n}{2}$ are connected by a unique geodesic. It follows that a distance-transitive group of $ \Ga $ is transitive on all geodesics of length $ \ell<\frac{n}{2} $. We thus only need to show that $ \Ga $ (resp.~$ \Ga^D $) is diameter-transitive. \vs

According to \cite[\S\,2, \S\,4]{GP-book}, $ \Ga $ is known as a {\it classical} one satisfying the {\textit{Moufang condition}}{\rm\,:} for any geodesic $ \gamma=(v_0,v_1,\ldots,v_{n-2}) $ in $ {\rm I}(\Ga) $ of length $ n-2 $ and $ v_{n-1} $ incident with $ v_{n-2} $, $ v_{n-1}\ne v_{n-3} $, the group of all \textit{$ \gamma $-elations}, i.e., the collineations of ${\rm I}(\Ga)$ which fixes all elements incident with at least one element of $ \gamma $, acts transitively on the set of elements incident with $ v_{n-1} $ but distinct from $ v_{n-2} $.\vs

Let $ \alpha=(u_{-1},u_0,\ldots,u_{n-3}) $, $ \beta=(u_0,u_1,\ldots,u_{n-2}) $ be geodesics in ${\rm I}(\Ga)$, where $ u_0 $ is a point (resp.~a line). By Moufang condition, the group of $ \alpha $-elations fixes $ (u_{-1},u_0,\ldots,u_{n-2}) $ and is transitive on geodesics of form $ (u_0,u_1,\ldots,u_{n-2},x) $\,; the group of $ \beta $-elations fixes each $ (u_0,u_1,\ldots,u_{n-2},x) $ and is transitive on geodesics of form $ (u_0,u_1,\ldots,u_{n-2},x,y) $. Thus, in particular, the collineation group of ${\rm I}(\Ga)$ contains a subgroup that fixes $ (u_0,u_1,\ldots,u_{n-2}) $ and is transitive on these diameters $ (u_0,u_1,\ldots,u_{n-2},x,y) $. Correspondingly, $ L:(u_0,u_2,\ldots,u_{n-2}) $ is a fixed geodesic of length $ \frac{n-2}{2} $ in $ \Ga $ (resp.\,$\Ga^D$), and $ \Aut\,\Ga $ (resp.\,$\Aut\,\Ga^D$) contains a subgroup that fixes $ L $ and is transitive on the set of diameters $ (u_0,u_2,\ldots,u_{\frac{n-2}{2}},y) $ obtained by extending $ L $. Then $ \Ga $ (resp.~$ \Ga^D $) is diameter-transitive. We are done.\vs

In conclusion, we have the following result.

\begin{pro}
Each of the distance-transitive generalized polygon with diameter $ \d\geqslant 3 $ is geodesic-transitive.
\end{pro}

\section{Polar Grassmann Graphs}\label{sec-pgg}
In this section we extend our study to polar Grassmann graphs. These graphs are derived from \textit{polar spaces} (for an axiomatic definition, see \cite[Sect.\,2.4]{pgg2}).\vs\vs\vs

\subsection{\it Graphs derived from W}{~}\vskip0.03in

Following the notation established in {\rm Section \ref{sec-polarspace}}, let $ W=(V,B) $ be one of the spaces of Witt index $ \o\geqslant 2 $ listed in {\rm Table \ref{tab-pspace}}, and denote by $ [\stack{W}{k}]_0 $ the set of singular $ k $-subspaces of $ W $ where $ 1\leqslant k\leqslant \o $. Then $ U\leqslant U^\perp $ for any $ U\in[\stack{W}{k}]_0 $. The space $ W $ defines a \textit{polar space} by considering $ [\stack{W}{1}]_0 $ and $ [\stack{W}{2}]_0 $, respectively, as the set of \textit{points} and \textit{lines}, and considering the containment relation as the \textit{incidence relation}. It is known that most polar spaces can be obtained in this way \cite{Tits}. We focus on graphs derived from such spaces.

\begin{defi}
The \textbf{polar Grassmann graph} $ \PG_{W}(k) $ is defined on $ [\stack{W}{k}]_0 $, where, if $k=\o$, then $ u,v $ are adjacent whenever $dim\ u\cap v= k-1${\rm\,;} if $ k<\o $, then $ u,v $ are adjacent whenever $u+v\in[\stack{W}{k+1}]_0$. In particular, $ \PG_{W}(1) $ is referred to as a \textbf{polar graph}, and $ \PG_{W}(\o) $ is referred to as a \textbf{dual polar graph}.
\end{defi}

The dual polar graph $ \PG_{W}(\o) $ was studied in Section \ref{sec-polarspace}. We thus only consider the graphs formed by non-maximal singular subspaces and assume that $ k<\o $. The description of automorphisms of these graphs is given in \cite[Sect.\,4.6]{pgg3} (see also \cite{pgg1})\,; such automorphisms are induced by collineations of the associated polar spaces. It follows by Theorem 8.6 in 
Tits \cite{Tits} that the full automorphism group is exactly the classical group $G_W $ given in Table \ref{tab-pspace}.\vs\vs\vs

\subsection{\it Distance and geodesics in $ \PG_{W}(k) $, $ k<\o $}{~}\vs
Let $ \Ga=\PG_{W}(k) $, $ k<\o $. The distance function of $ \Ga $ is given in \cite[Lem.\,1]{pgg1}. Here, we demonstrate it by revealing the geodesic form. Note that any two vertices are adjacent in $ \Ga $ only if their intersection has dimension $ k-1 $. Hence $ \Ga $ is regarded as a subgraph of the Grassmann graph $ \G_V(k) $ obtained by deleting some of the vertices and the edges. It follows that for any vertices $ X,Y $ of $ \Ga $, \[\d_\Ga(X,Y)\geqslant\d_{\G_V(k)}(X,Y)=k-dim\ X\cap Y.\] There are two cases\,: (1) $ \d_\Ga(X,Y)=\d_{\G_V(k)}(X,Y) $\,; (2) $ \d_\Ga(X,Y)>\d_{\G_V(k)}(X,Y) $. \vs\vs

In case (1), any geodesic $ L $ in $ \Ga $ from $ X $ to $ Y $ is realized as a geodesic in $ \G_V(k) $. Geodesics in $ \G_V(k) $ are described in Section \ref{sec-prospace}. It follows that we can take a basis $ \eta=\{w_1,\ldots,w_{k-m}\} $ of $ X\cap Y $, and extend it to a basis $\eta_X=\eta\cup\{x_1,\ldots,x_m\} $ of $ X $ and a basis $\eta_Y=\eta\cup\{y_1,\ldots,y_m\} $ of $ Y $ such that $ L $ is determined as \vs\vs
\hskip 0.5in $L:(X_0=X,\,X_1,\ldots,X_{m-1},\,X_m=Y) $, where for $ 0\leqslant i \leqslant m $,\vs
\hskip 0.5in $X_i=\l x_1,\ldots x_{m-i},\,y_{m-i+1},\ldots,y_{m},\,w_1,\ldots,w_{k-m}\r$.\hskip1in({\rm F1})\vs\vs
\noindent Now, $ L $ being a geodesic in $ \Ga $ implies that $ X_i+X_{i+1}\in[\stack{W}{k+1}]_0 $, or equivalently, the set of vectors $ \eta_X \cup\eta_Y $ satisfies the following relation\,:\vs\vs
\hskip 0.5in $ B(x_i,y_j)=0 $ for $ 1\leqslant i\leqslant j\leqslant m $.\hskip 2.23in({\rm R1})\vs\vs

By a well-known property, the two singular subspaces $ X,Y $ are spanned by subsets of some polar frame. It is then verified directly that if and only if in the following case there exist no bases $ \eta_X $, $ \eta_Y $ satisfying relation (R1)\,: 
\begin{itemize}
\item[\bf(i)] for each $ x\in X\setminus Y $, there is $ y\in Y\setminus X $ such that $ B(x,y)\ne 0${\rm\,;} for each $ y\in Y\setminus X $, there is $ x\in X\setminus Y $ such that $ B(x,y)\ne 0${\rm\,;}\vskip1pt
\item[\bf(ii)] or equivalently, there is a basis $ \{\alpha_1,\ldots,\alpha_m\} $ of $ X/X\cap Y $ and a basis $ \{ \beta_1,\ldots,\beta_m\} $ of $Y/X\cap Y$ such that $ B(\alpha_i,\beta_j)=\delta_{ij} $.\vs\vs
\end{itemize}

In case {\bf(i)} (or equivalently {\bf(ii)}), none of geodesics in $ \G_V(k) $ from $ X $ to $ Y $ can be realized as a geodesic in $ \Ga $. Hence this case falls into case (2) where $ \d_\Ga(X,Y)> \d_{\G_V(k)}(X,Y) $. Note that any vertex $ X' $ of $ \Ga $ adjacent to $ X $ intersects $ Y $ in a subspace of dimension no greater than the dimension of $ X\cap Y $. Then \[\d_\Ga(X',Y)\geqslant\d_{\G_V(k)}(X',Y)\geqslant \d_{\G_V(k)}(X,Y).\] The equality holds only for such $ X' $ that $ X', Y $ do not satisfy condition {\bf(i)}, and meanwhile, $ X'\cap Y=X\cap Y $. Such vertices $ X' $ exist in the following way\,: $ X'=X\cap X'+U $, where the intersection $ X\cap X' $ is a $ (k-1) $-subspace of $ X $ which contains $ X\cap Y $, and $ U $ is a singular $ 1 $-subspace of $(X+Y)^\perp$ which is not contained in $X+Y $. We then have $ \d_\Ga(X',Y)= \d_{\G_V(k)}(X,Y) $, and so $$\d_\Ga(X,Y)= \d_{\G_V(k)}(X,Y)+1. $$ \vs

Now, for any geodesic $ L:(X,X',\ldots,Y) $, the sub-geodesic $ L' $ from $ X' $ to $ Y $ is in case (1). Extend a basis $ \eta=\{w_1,\ldots,w_{k-m}\} $ of $ X'\cap Y $, respectively, to a basis $ \eta\cup\{x_0,\ldots,x_{m-1}\} $ of $ X'$ and  a basis $ \eta\cup\{y_1,\ldots,y_m\} $ of $ Y $ such that\vs\vs
\hskip 0.5in $ L':(X'=X_0,\,X_1,\ldots,X_{m-1},\,X_m=Y) $, where for  $ 0\leqslant i \leqslant m $,\vs
\hskip 0.5in $X_i=\l x_0,\,x_1,\ldots x_{m-i-1},\,y_{m-i+1},\ldots,y_{m},\,w_1,\ldots,w_{k-m}\r $, \hskip0.5in({\rm F2.1})\vs
\noindent while the following relation is satisfied\,:\vs\vs
\hskip 0.5in $ B(x_i,y_j)=0 $,\, for $ 0\leqslant i<j\leqslant m $.\hskip 2.25in({\rm $\ast$})\vs\vs

\noindent In particular, $ B(x_0,Y)=0$. Thus, $ x_0\notin X $, and so $ X'=\l x_0\r+X\cap X' $. By redefining each of $ x_{1},\ldots,x_{m-1} $ as itself plus a multiple of $ x_{0} $ if necessary, we can let $ X\cap X'=\l x_1,\ldots,x_{m-1},\,w_1,\ldots,w_{k-m}\r$, and let \vs\vs
\hskip 0.5in $X=\l x_1,\ldots,x_{m-1},x_m,\,w_1,\ldots,w_{k-m}\r$,\, where $ x_m\in X\setminus X' $. \hskip 0.16in({\rm F2.2})\vs\vs

\noindent Further, since $ X,Y $ satisfy condition {\bf(i)}, it follows by ($ \ast $) that $ B(x_i,y_i)\ne 0 $ for each $ i $ from $ m $ to $ 1 $. After a process of normalization which redefines each $ y_i $ from $y_{m-1}$ to $ y_1 $ as certain one lying in $\l y_i,y_{i+1},\ldots,y_m\r$, we can require that\vs\vs
\hskip 0.5in $B(x_i,y_j)=\delta_{ij}$\, for $ 1\leqslant i,j\leqslant m $. \hskip2.2in({\rm R2})\vs\vs

In conclusion, we have the following result.

\begin{defi}
Two vertices $ {X,Y} $ are called \textbf{opposite}, if they satisfy condition {\bf(i)} {\rm(}or equivalently, {\bf(ii)}{\rm)} stated above.
\end{defi}

\begin{propo}\label{prop-pg-geo}
Let $ X,Y $ be two vertices of $ \Ga=\PG_{W}(k) $, $ k<\o $.
\begin{itemize}
\item[{\rm (1)}] If $ X,Y $ are not opposite, then $ \d(X,Y)=k-dim\ X\cap Y $, and a geodesic from $ X $ to $ Y $ is of form {\rm (F1)} with relation {\rm (R1)} satisfied.\vs
\item[{\rm (2)}] If $ X,Y $ are opposite, then $ \d(X,Y)=k-dim\ X\cap Y +1 $, and a geodesic from $ X $ to $ Y $ is of form {\rm (F2.1)}\,-\,{\rm (F2.2)} with relation {\rm (R2)} satisfied.\vs
\end{itemize}
\noindent In particular, the graph $ \Ga $ is of diameter $ k+1 $.
\end{propo}

\begin{corol}\label{cor-notdr}
For $ 1< k<\o $, the graph $ \Ga=\PG_{W}(k) $ is not distance-regular.
\end{corol}

\begin{proof}
Let $ X=\l x_1,\ldots,x_{k}\r $ be a vertex of $ \Ga $. Expand these vectors to a polar frame $\{x_1,\ldots,x_{k},\,x_{k+1},\ldots;\,y_1,\ldots,y_{k},\,y_{k+1},\ldots\} $. The following vertices \[\mbox{$X_1=\l x_1,\ldots,x_{k-1},\,y_{k}\r$ and $X_2=\l x_1,\ldots,x_{k-2},\,y_{k},y_{k+1}\r$}\] are both at distance $ 2 $ from $ X $. Let $ Y_1 $ be a vertex adjacent to $ X_1 $. Then $ X_1\cap X $, $ X_1\cap Y_1 $ are two hyperplanes of $ X_1 $ which intersect in a subspace of dimension no less than $ k-2 $. It follows that, either $ X\cap Y_1 $ is a $ (k-1) $-space\,; or it is a $ (k-2) $-space contained in $ X_1\cap X $, so that, $ X,Y_1 $ are not opposite. In both cases, $ \d_\Ga(X,Y_1)\leqslant 2 $. On the other hand, $ Y_2=\l x_1,\ldots,x_{k-2},\,y_{k-1},y_{k}\r $ is a vertex adjacent to $ X_2 $ but at distance $ 3 $ from $ X $. Thus, $ \Ga $ is not distance-regular.
\end{proof}\vs

\subsection{\it Orbits of geodesics}{~}\vs

Let $ \Ga=\PG_{W}(k) $, $ k<\o $. We determine the $ \Aut\,\Ga $-orbits of geodesics in $ \Ga $. For convenience, we call a geodesic \textit{opposite} if its two endpoints are opposite\,; otherwise, call it \textit{non-opposite}. Recall $ \Aut\,\Ga=G_W $. It is clear that opposite geodesics and non-opposite geodesics lie in different $ \Aut\,\Ga $-orbits.\vs

\begin{propo}\label{prop-pg-op}
For $2\leqslant\ell\leqslant k+1 $, the opposite geodesics of length $ \ell $ form a single $ \Aut\,\Ga $-orbit. 
\end{propo}

\begin{proof}
Let $ L $ be an opposite geodesic of length $ \ell=m+1 $. By proposition \ref{prop-pg-geo}, there is a set of vectors $ M=\{x_0,x_1,\ldots,x_m,\,y_1,\ldots,y_m,\,w_1,\ldots,w_{k-m}\} $ which induces $ L $ in form {\rm (F2)} and satisfies relation {\rm (R2)}. Note that the relation (R2) is \textit{complete}, i.e., for any $ u,v\in M $, the value $ B(u,v) $ is determined as either $ B(u,v)= 0 $, or $ \{u,v\}=\{x_i,y_i\} $, $ B(u,v)=1 $. Consequently, $ M $ can be expanded to a polar frame $ N $ of $ W $ in a fixed way\,:
$$ N=\{x_1,x_2,\ldots,x_m,\,x_0,w_1,\ldots,w_{k-m},\cdots\,;\,y_1,y_2,\ldots,y_m,\,y_0,z_1,\ldots,z_{k-m},\cdots\}. $$
	
Now, any other opposite geodesic $ \widetilde{L} $ of length $ \ell $ is induced by some set of vectors $ \widetilde{M} $ which can be expanded to a polar frame $ \widetilde{N} $ in a similar manner. By Witt's Lemma, there is $ g\in G_W $ which maps the frame $ N $ to $ \widetilde{N} $, by mapping each $ v\in N $ to the corresponding $ \widetilde{v}\in \widetilde{N} $. In particular, $ g $ maps $ M $ to $ \widetilde{M} $ and $ L $ to $ \widetilde{L} $. 
\end{proof}\vs\vs

We then consider non-opposite geodesics. Let $ L:(X_0,X_{1},\ldots,X_{m}) $ be such a geodesic, where $ 1\leqslant m\leqslant k $. By proposition \ref{prop-pg-geo}, there is a set of vectors $ M=\{x_1,\ldots,x_m,\,y_1,\ldots,y_m,\,w_1,\ldots,w_{k-m}\} $ which induces $ L $ in form (F1) and satisfies relation ({\rm R1})\,: $ B(x_i,y_j)=0 $ for $ 1\leqslant i\leqslant j\leqslant m $.\vs

Note that (R1) is not as complete as (R2), since $B(x_i,y_k)$ for  $1\leqslant k<i\leqslant m$ is not yet determined. After a process of normalization which redefines each $ y_i $ as certain one in $ \l y_i,y_{i+1},\ldots,y_m\r $ and each $ x_i $ as certain one in $ \l x_0,x_1,\ldots,x_i\r $, we can rewrite (R1) as follows\,: for each $ x_i $, $ 1\leqslant i\leqslant m $, either\vs\vs

\hskip0.1in $B(x_i,Y)=0$\,; or $ B(x_i,y_{i'})=1 $, for $ i'<i $ and $ B(x_i,y_k)=0 $ for $ k\ne i' $, where\vs
\hskip0.1in if $ B(x_i,y_{i'})=1=B(x_j,y_{j'}) $ and $i\ne j$, then $i'\ne j' $.\hskip1.5in (RI)\vs
\noindent 

\begin{pro}\label{prop-pg-nonop}
The non-opposite geodesics admitting the same set of relations in {\rm (RI)} form a single $ \Aut\,\Ga $-orbit.
\end{pro}

\begin{proof}
It is similar to Proposition \ref{prop-pg-op} to prove that geodesics admitting the same relations lie in the same $ \Aut\,\Ga $-orbit. On the other hand, let $ L $, $\widetilde{L}$ be two geodesics induced, respectively, in form (F1) by the sets of vectors\vs

$ M=\{x_1,\ldots,x_m,\,y_1,\ldots,y_m,\,w_1,\ldots,w_{k-m}\} $, and \vs

$ \widetilde{M}=\{\widetilde{x}_1,\ldots,\widetilde{x}_m,\,\widetilde{y}_1,\ldots,\widetilde{y}_m,\,\widetilde{w}_1,\ldots,\widetilde{w}_{k-m}\}  $. \vs

\noindent Suppose $ M, \widetilde{M} $ satisfy different relations. Without loss of generality, let $ a $ be the minimal subscript such that\vs
\hskip1in $B(x_a,y_b)=1=B(\widetilde{x}_a,\widetilde{y}_c)$, where $b>c$\,; or\vs
\hskip1in $B(x_a,y_b)=1$, while $B(\widetilde{x}_a,\widetilde{y}_k)=0$, for all $\widetilde{y}_k\in \widetilde{M}$.\vs
\noindent Note that $ b,c<a $. Consider the following vertices lying in $ L $ or $ \widetilde{L} $\,:\vs
$ X_{m-a}=\l x_1,\ldots,x_a,y_{a+1},\ldots,y_m,\cdots\r $,\ \ $ \widetilde{X}_{m-a}=\l \widetilde{x}_1,\ldots,\widetilde{x}_a,\widetilde{y}_{a+1},\ldots,\widetilde{y}_m,\cdots\r $,\vs
$X_{m-b+1}=\l x_1,\ldots,x_{b-1},y_{b},\ldots,y_m,\cdots\r $,\ \ $\widetilde{X}_{m-b+1}=\l \widetilde{x}_1,\ldots,\widetilde{x}_{b-1},\widetilde{y}_{b},\ldots,\widetilde{y}_m,\cdots\r $.\vs
\noindent Then, in both cases, we have \[\begin{array}{cl}
& \dim\, (X_{m-a}\cap X_{m-b+1}^{\perp})-\dim\,(X_{m-a}\cap X_{m-b+1})\\
= & \dim\, (\widetilde{X}_{m-a}\cap \widetilde{X}_{m-b+1}^{\perp})-\dim\,(\widetilde{X}_{m-a}\cap \widetilde{X}_{m-b+1})-1.
\end{array}\] If there is an element in $ \Aut\,\Ga=G_W $ which maps $ L$ to $ \widetilde{L} $, then, respectively, it maps $ X_{m-a}$, $X_{m-b+1}$ and $X_{m-b+1}^{\perp} $ to $ \widetilde{X}_{m-a} $, $ \widetilde{X}_{m-b+1} $ and $\widetilde{X}_{m-b+1}^\perp $. We thus get a contradiction. Hence $ L, \widetilde{L} $ lie in different $ \Aut\,\Ga $-orbits.
\end{proof}\vs

We then calculate the number of these orbits, or equivalently, how many ways there are to determine a set of relations in (RI). We divide geodesics into different types. The {\it type} of $ L:(X=X_0,X_{1},\ldots,X_{m}=Y) $ is defined to be the tuple $ \t= (t_{1},\ldots,t_{m})$ where each $$ t_{i}:=\dim\, (X_m\cap X_{m-i}^{\perp})-\dim\,(X_m\cap X_{m-i}). $$ 

Recall that $ L $ is induced by $ M=\{x_1,\ldots,x_m,\,y_1,\ldots,y_m,\,w_1,\ldots,w_{k-m}\} $ as\vs
\hskip 0.5in $X_i=\l x_1,\ldots x_{m-i},\,y_{m-i+1},\ldots,y_{m},\,w_1,\ldots,w_{k-m}\r$.\hskip1.2in({\rm F1})\vs
\noindent It is straightforward to verify that the type $ \t $ of $ L $ is a non-decreasing sequence satisfying that\,:\vs
\begin{itemize}
\item[\bf(\t.1)]\ $ t_{1}=1 $\,; since $ B(x_1,y_k)=0 $ for all $ y_k\in M $.\vs
\item[\bf(\t.2)]\ $ t_{i}=t_{i-1}+1 $ or $ t_{i-1} $\,; while the former holds if $ B(x_i,y_{k})=0 $ for all $ y_k\in M $, and the latter holds if $ B(x_{i},y_{i'})=1 $ for some $ y_{i'} $ with $ i'<i$.\vs
\item[\bf(\t.3)]\ $ t_{i}\leqslant \o-k $\,; as $ (X_m\cap X_{m-i}^{\perp})+X_{m-i} $ is a singular $(t_i+k) $-subspase.
\end{itemize}\vs

\noindent In other word, each given type $ \t=(t_{1},\ldots,t_{m}) $ corresponds to a choice of such vectors $ x_i\in M $ satisfying that $ B(x_i,y_k)=0 $ for all $ y_k\in M $. Note that $ \t $, as a non-decreasing sequence starting from $ 1 $, is determined by the increasing items. Let $ 1\leqslant a_1< a_2<\cdots< a_{t_{m}-1}<a_{t_m}=m $ be a sequence of subscripts such that \[\mbox{$t_{a_{s}}=s$\,  and\,  $t_{a_{s}+1}=s+1$.}\] That is, $ t_{a_{s}+1} $ is an increasing item. Let $I= \{a_{1}+1,a_{2}+1\ldots,a_{t_m-1}+1\}$ be the set of subscripts of these increasing items. Then\,:\vs
\hskip1in for $ i\in I\cup\{1\} $,\,   $B(x_i,y_k)=0$ for all $ y_k\in M $\,;\vs
\hskip1in for $ i\notin I\cup\{1\} $,\,   $ B(x_i,y_{i'})=1 $ for some $ i'<i $.\vs\vs

Further, for each $ a $ with $ 1<a\leqslant a_1 $, $ B(x_{a},y_{a'})=1 $ for some $ y_{a'}\in\{y_1,\ldots,y_{a-1}\} $, while $ a-2 $ vectors among them have already been chosen by $ x_b $ (with $1<b<a$), so there is only $ 1 $ choice for $ y_{a'} $. Similarly, for $ 2\leqslant s\leqslant t_m $ and each $ a $ with $ a_{s-1}+1<a\leqslant a_{s} $, $ B(x_{a},y_{a'})=1 $ for some $ y_{a'} \in \{y_1,\ldots,y_{a-1}\} $, where $ a-s-1 $ vectors among them have been chosen by $ x_b $ (with $1<b<a$, and $b\notin\{a_1+1,a_2+1,\ldots,a_{s-1}+1\} $), so there are $ s $ choices for $ y_{a'} $. In conclusion, the type $ \t= (t_{1},\ldots,t_{m}) $ give rise to $ c_{\t} $ different sets of relations in (RI), where $$ c_{\t}= 1^{(a_1-1)}\times 2^{(a_2-a_1)-1}\times\cdots\times t_m^{(a_{t_m}-a_{t_{m-1}})-1}=(\prod _{i=1}^{m}t_{i})/(t_{m}!). $$

\begin{propo}\label{prop-pg-orbitsnum}
For $ 1\leqslant m\leqslant k $, the non-opposite geodesics of length $ m $ are divided into $ \mathcal{L}(m) $ distinct $ \Aut\,\Ga $-orbits, where $ \mathcal{L}(m)=\sum\limits_\t c_\t $. In particular, if $ m\leqslant\o-k $, then $\mathcal{L}(m)$ is equal to the Bell number ${\rm B}(m)$.
\begin{table}[hbt]
\centering
\caption{Bell numbers ({\sf OEIS}, A000110)}
\scalebox{0.8}{
\begin{tabular}{c|c|c|c|c|c|c|c|c}
\toprule[1pt]
${\rm B}(1)$ & ${\rm B}(2)$ & ${\rm B}(3)$ & ${\rm B}(4)$ & ${\rm B}(5)$ & ${\rm B}(6)$ & ${\rm B}(7)$ & ${\rm B}(8)$ & $\cdots$\\
\midrule[0.5pt]
			
$ 1 $ & $ 2 $ & $ 5 $ & $ 15 $ & $ 52 $ & $ 203 $ & $ 877 $ & $ 4140$ & $ \cdots $\\
\bottomrule[1pt]
\end{tabular}}
\end{table}
\end{propo}

\begin{proof}
It follows by Proposition \ref{prop-pg-nonop} that $ \mathcal{L}(m)=\sum\limits_\t c_\t $, where $ \t=(t_1,t_2,\ldots,t_m) $ ranges over all possible types, namely those satisfying conditions {\bf (\t.1)}\,-{\bf(\t.3)}.\vs
	
Suppose further that $ m\leqslant\o-k $. Then $ t_m\leqslant m\leqslant\o-k $, so that, the condition {\bf(\t.3)} is always satisfied. Then $ \t $ ranges over the sequences satisfying\,: \[\mbox{$t_1=1$,\ \  $t_{i+1}=t_{i}$ or $t_{i}+1$.}\] In this case, we claim that $ \mathcal{L}(m)=\sum\limits_\t c_\t $ is equal to the number of distinct rhyme schemes of length $ m $. A \textit{rhyme scheme} is a string of numbers such that the leftmost number is always `$ 1 $' and no number may be greater than one more than the greatest number to its left. Note that each given type $$ \t=(1,1,\ldots,1,\,2,2,\ldots,2,\,3,\ldots,s-1,\,s,s,\ldots,s) $$ exactly gives rise to $ c_{\t} $ distinct rhyme schemes in the following way\,:
$${\sf r}(\t)=\{(1,n_{11},\ldots,n_{1m_1},\,2,n_{21},\ldots,n_{2m_2},\,3,\ldots,s-1,\,s, n_{s1},\ldots,n_{sm_s})\}, $$ where each $ n_{i\ast}\in\{1,\ldots,i\} $ has $ i $ choices. Further, the positions of the increasing items in $ \t $ correspond exactly to the positions where the first `2', `3', `4', $\ldots$ appear in each of the rhyme schemes lying in $ {\sf r}(\t) $. Hence different types give rise to different rhyme schemes. As $ \t $ runs over the above sequences, it gives rise to all possible rhyme schemes. Then the claim holds.\vs
	
At last, we note that the number of distinct rhyme schemes of length $ m $ is known as the \textit{Bell number} ${\rm B}(m)$. These numbers can be defined in many different ways\,; see {\sf OEIS}, A000110.
\end{proof}\vs

\begin{exam}
{\rm (1)} Let $ W $ be a space with Witt index $\o\geqslant 8 $, and let $ \Ga_1=\PG_{W}(4) $. The types of non-opposite geodesics are shown as follows (a specific example is highlighted in red), where the numbers of $ \Aut\,\Ga $-orbits on these geodesics are $\mathcal{L}(1)=1$,\ $\mathcal{L}(2)=1\cdot 2=2$,\ $\mathcal{L}(3)=1\cdot 3+2=5$, $\mathcal{L}(4)=1\cdot 4+2\cdot 2+3+4=15$.
\begin{figure}[htb]
\scalebox{0.6}{
\begin{tikzpicture}
	\node[anchor= east] at (-0.3,0) {\large Types $ \t: $};
	
    \draw[line width= 1pt] (0,0)--(4,0);
	\draw[line width= 1pt] (1.5,-1)--(3.5,-1);
	\draw[line width= 1pt] (2,-2)--(3,-2);
	\draw[line width= 1pt] (1,0)--(2.5,-3);
	\draw[line width= 1pt] (2,0)--(3,-2);
	\draw[line width= 1pt] (3,0)--(3.5,-1);
		    	
	\foreach \x in {0,1,2,3,4} \filldraw[fill=black] (\x,0) circle (2pt);
	\foreach \x in {1,2,3,4} 
	\node[anchor= north east] at (\x,0) {$1$};
	\foreach \x in {1.5,2.5,3.5} \filldraw[fill=black] (\x,-1) circle (2pt);
	\foreach \x in {1.5,2.5,3.5} 
	\node[anchor= north east] at (\x,-1) {$2$};
	\foreach \x in {2,3} 
	\filldraw[fill=black] (\x,-2) circle (2pt);
	\foreach \x in {2,3} 
	\node[anchor= north east] at (\x,-2) {$3$};
	\filldraw[fill=black] (2.5,-3) circle (2pt);
	\node[anchor= north east] at (2.5,-3) {$4$};
	    	    
	\draw[line width=1pt,color=red] (0,0) -- (2,0);	
	\draw[line width=1pt,color=red] (2,0) -- (3,-2);
	\node[anchor= north west] at (3.2,-2) {\textcolor{red}{$ \t= (1,1,2,3) $}};
	
	\node at (12,-1)
	{\large $c_{\t}=
	\begin{cases}
	\,2^{1}, & \t=(1,2,2),\, (1,1,2,2),\,(1,2,2,3)\,;\\
	\,3^{1}, & \t=(1,2,3,3)\,;\\
	\,2^{2}, & \t=(1,2,2,2)\,;\\
	\,1, & \mbox{for\ remaining \t}.
	\end{cases} $};	
	
\end{tikzpicture}}
\end{figure}

\noindent {\rm(2)} Let $ W $ be a space with Witt index $\o=6 $, and let $ \Ga_2=\PG_{W}(4) $. The types are shown as follows, where the numbers of $ \Aut\,\Ga $-orbits on these geodesics are 
$\mathcal{L}(1)=1$,\ $\mathcal{L}(2)=1\cdot 2=2$,\ $\mathcal{L}(3)=1\cdot 2+2=4$, $\mathcal{L}(4)=1\cdot 2+2+4=8.$
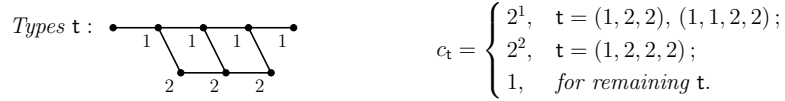
\begin{figure}[htb]
\scalebox{0.6}{
\begin{tikzpicture}
	\node[anchor= east] at (-0.3,0) {\large Types $ \t: $};
	
    \draw[line width= 1pt] (0,0)--(4,0);
	\draw[line width= 1pt] (1.5,-1)--(3.5,-1);
	\draw[line width= 1pt] (1,0)--(1.5,-1);
	\draw[line width= 1pt] (2,0)--(2.5,-1);
	\draw[line width= 1pt] (3,0)--(3.5,-1);
			
	\foreach \x in {0,1,2,3,4} \filldraw[fill=black] (\x,0) circle (2pt);
	\foreach \x in {1,2,3,4} 
	\node[anchor= north east] at (\x,0) {$1$};
	\foreach \x in {1.5,2.5,3.5} \filldraw[fill=black] (\x,-1) circle (2pt);
	\foreach \x in {1.5,2.5,3.5} 
	\node[anchor= north east] at (\x,-1) {$2$};
	
	\node at (11,-0.5) 
	{\large $c_{\t}=
	\begin{cases}
	\,2^{1}, & \t=(1,2,2),\, (1,1,2,2)\,;\\
	\,2^{2}, & \t=(1,2,2,2)\,;\\
	\,1, & \mbox{for\ remaining \t}.
    \end{cases} $};		
\end{tikzpicture}}
\end{figure}
\end{exam}

We conclude with the {\it proof of \,{\rm Theorem \ref{thm-pg}}}\,:\vs
(i) $\Rightarrow$ (ii)\,: obvious.\vs
(ii) $\Rightarrow$ (iii)\,: Corollary \ref{cor-notdr}.\vs
(iii) $\Rightarrow$ (i)\,: a dual polar graph is geodesic-transitive (Proposition \ref{prop-dualpolar})\,; meanwhile, a polar graph is of diameter $ 2 $, where each $ 1 $-geodesic is with non-opposite ends and each $ 2 $-geodesic is with opposite ends (Proposition \ref{prop-pg-geo}), so it is geodesic-transitive by Propositions \ref{prop-pg-op} and \ref{prop-pg-orbitsnum}.\qed

%\section*{Acknowledgments} The author is sincerely grateful to Prof. C.H. Li of Southern University of Science and Technology of China, and to his Ph.D. supervisor Prof. R.Q. Feng of Peking University. They both offered invaluable support in all respects.

\section*{Declarations}
\begin{center}
{Funding and/or Conflicts of interests/Competing interests}
\end{center}

The author(s) declares that there is no any financial or personal relationship with other people or organizations not mentioned that can inappropriately influence the work.

\end{document}